\date{}
\documentclass [11pt]{article}
\usepackage{amsfonts,amsmath,amssymb,graphicx,amsthm,comment,psfrag}

\title{\vspace{-1cm}Tur\'an numbers of bipartite graphs plus an odd cycle}
\author{ Peter Allen\thanks{ Department of Mathematics, London School of
Economics, UK, p.d.allen@lse.ac.uk. Research supported in part by FAPESP
(Proc.~2010/09555-7).} \and Peter Keevash\thanks{ School of Mathematical Sciences, Queen Mary University of London, UK,
p.keevash@qmul.ac.uk. Research supported in part by ERC grant 239696 and EPSRC
grant EP/G056730/1.} \and Benny Sudakov\thanks{ Department of Mathematics,
University of California, Los Angeles, CA, USA, bsudakov@math.ucla.edu. Research
supported in part by NSF grant DMS-1101185, by AFOSR MURI grant FA9550-10-1-0569 and by a USA-Israel BSF
grant.} \and Jacques Verstra\"ete\thanks{ Department of
Mathematics, University of California, San Diego, CA, USA,
jverstra@math.ucsd.edu. Research supported by NSF Grant DMS-110 1489.} }

\newtheorem{thm}{Theorem}[section] 
\newtheorem{prop}[thm]{Proposition}
\newtheorem{lem}[thm]{Lemma}

\newtheorem{conj}[thm]{Conjecture}

\newtheorem{prob}[thm]{Problem}
\newcounter{foo}
\newtheorem{definition}[foo]{Definition}

\newcommand{\mc}[1]{\mathcal{#1}}
\newcommand{\mb}[1]{\mathbb{#1}}

\newcommand{\eps}{\varepsilon} 
\newcommand{\sub}{\subseteq}
\newcommand{\ex}[2]{\text{ex}\big({#1},{#2}\big)}
\newcommand{\z}[2]{\text{z}\big({#1},{#2}\big)}
\newcommand{\zz}[3]{\text{z}\big({#1},{#2},{#3}\big)}

\def\COMMENT#1{}

\begin{document}

\maketitle

\begin{abstract}
For an odd integer $k$, let $\mathcal{C}_k = \{C_3,C_5,\dots,C_k\}$ denote the
family of all odd cycles of length at most $k$ and let $\mathcal{C}$ denote the
family of all odd cycles. Erd\H{o}s and Simonovits~\cite{ESi1} conjectured that
for every family $\mathcal{F}$ of bipartite graphs, there exists $k$ such that
$\ex{n}{\mathcal{F} \cup \mathcal{C}_k} \sim \ex{n}{\mathcal{F} \cup
\mathcal{C}}$ as $n \rightarrow \infty$. This conjecture was proved by Erd\H{o}s
and Simonovits when $\mathcal{F} = \{C_4\}$, and for certain families of even
cycles in~\cite{KSV}. In this paper, we give a general approach to the
conjecture using Scott's sparse regularity lemma. Our approach proves the
conjecture for complete bipartite graphs $K_{2,t}$ and $K_{3,3}$: we obtain more
strongly that for any odd $k \geq 5$, \[ \ex{n}{\mathcal{F} \cup \{C_k\}} \sim
\ex{n}{\mathcal{F} \cup \mathcal{C}}\] and we show further that the extremal
graphs can be made bipartite by deleting very few edges. In contrast, this
formula does not extend to triangles -- the case $k = 3$ -- and we give an
algebraic construction for odd $t \geq 3$ of $K_{2,t}$-free $C_3$-free graphs
with substantially more edges than an extremal $K_{2,t}$-free bipartite graph on
$n$ vertices. Our general approach to the Erd\H{o}s-Simonovits conjecture is
effective based on some reasonable assumptions on the maximum number of edges in
an $m$ by $n$ bipartite $\mathcal{F}$-free graph.
\end{abstract}

\section{Introduction}

Let $\mathcal{F}$ be a family of graphs. Then we say that a graph is
\emph{$\mathcal{F}$-free} if it contains no member of $\mathcal{F}$ as a (not
necessarily induced) subgraph. We define the \emph{Tur\'an number}
$\ex{n}{\mathcal{F}}$ to be the maximum number of edges possible in a
$\mathcal{F}$-free graph on $n$ vertices. When $\mc{F}=\{F\}$ consists of a single forbidden graph we denote the Tur\'an number by $\ex{n}{F}$.
A classical theorem of Tur\'an \cite{TuranThm} gives an exact result for $\ex{n}{K_r}$, where
$K_r$ is the complete graph on $r$ vertices: the unique largest $K_r$-free graph
on $n$ vertices is the complete $(r - 1)$-partite graph with part sizes as equal
as possible. In general, Erd\H{o}s, Stone and Simonovits \cite{ErdSto,ErdSimExt}
showed that
$\ex{n}{\mc{F}} \sim (1-1/r)\binom{n}{2}$, where $r = \min \{ \chi(F)-1: F \in \mc{F}\}$.
Here we write $f(n) \sim g(n)$ for functions $f,g : \mathbb N \rightarrow \mathbb R$ if $\lim_{n \rightarrow \infty} f(n)/g(n) = 1$.
This determines the Tur\'an number asymptotically when $\mc{F}$ consists only of non-bipartite graphs. Furthermore,
Erd\H{o}s and Simonovits~\cite{ErdSimX} proved that the error term is zero if the family
$\mathcal{F}$ contains an edge-critical $(r+1)$-chromatic graph (i.e. one from
which one edge can be removed to decrease the chromatic number). For more
general families $\mathcal{F}$ the error term is necessary, and the complete
balanced $r$-partite graph is not extremal. In that case, it turns out that the
error term is controlled by Tur\'{a}n numbers for bipartite graphs~\cite{Simonovits}. The major open problem in the area, then, is to determine the behavior of
$\ex{n}{\mathcal{F}}$ for families $\mathcal{F}$ including bipartite graphs.

\subsection{Tur\'{a}n Numbers for Bipartite Graphs}

The problem of determining $\ex{n}{\mathcal{F}}$ when $\mathcal{F}$ contains bipartite graphs seems in general to be very hard: in almost all cases we do not know the order of magnitude of $\ex{n}{\mathcal{F}}$. In a seminal paper~\cite{ESi1}, Erd\H{o}s and Simonovits
made a number of broad conjectures on Tur\'{a}n numbers for bipartite graphs. In
particular
(see Conjecture 2 in~\cite{ESi1}) they made the following conjecture:

\begin{conj}\label{firstconj}
For every finite nonempty family $\mathcal{F}$ of graphs, there exist $\sigma,\alpha \geq 0$ such that
\[ \lim_{n \rightarrow \infty}  \frac{\ex{n}{\mathcal{F}}}{n^{\alpha}} = \sigma.\]
\end{conj}

 Erd\H{o}s and Simonovits~\cite{ESi1} stated that $\alpha$ should furthermore be
 rational, in agreement with all the cases where $\alpha$ is known to exist.
  When $\alpha$ exists, we refer to $\alpha$ as the {\em exponent} of
  $\mathcal{F}$. The lack of good constructive lower bounds for
  $\ex{n}{\mathcal{F}}$ is a major impediment to determining exponents of
  families of bipartite graphs. It appears unlikely that the order of
  magnitude of $\ex{n}{F}$,
   let alone the existence of the exponent of $F$, will ever be determined for every bipartite graph
   $F$. Even for specific graphs such as the complete bipartite graph $F =
   K_{4,4}$, the cycle of length eight $F = C_8$, and the three-dimensional
   cube graph $F = Q_3$, the order of magnitude of $\ex{n}{F}$ is not known.

\subsection{A conjecture of Erd\H{o}s and Simonovits}

In this paper, we study a conjecture of Erd\H{o}s and
Simonovits~\cite{ESi1} concerning Tur\'{a}n numbers for families $\mathcal{F}$ containing bipartite graphs
and odd cycles. Let $\mathcal{C}_k$ be the family of all odd cycles of length at most
$k$, and let $\z{n}{\mathcal{F}}$ denote the maximum size of a bipartite $n$-vertex $\mathcal{F}$-free graph. The general theme is that extremal $\mathcal{F}$-free graphs should be near-bipartite if $\mathcal{F}$ contains a long enough odd cycle as well as bipartite graphs. Precisely, Erd\H{o}s and Simonovits made the following conjecture (Conjecture 3 in~\cite{ESi1}):

\begin{conj}\label{ESC}
  Given any finite family $\mathcal{F}$ of graphs, there exists $k$ such that as $n \rightarrow \infty$,
  \[\ex{n}{\mathcal{F} \cup \mathcal{C}_k} \sim \z{n}{\mathcal{F}}.\]
\end{conj}

This is certainly true from the Erd\H{o}s-Stone-Simonovits Theorem~\cite{ErdSimX} when $\mathcal{F}$ consists only of non-bipartite graphs, moreover it is shown that the extremal
$\mathcal{F} \cup \{C_k\}$-free graphs are exactly bipartite if $n$ is large enough~\cite{ErdSimX}. Conjecture~\ref{ESC} is proved in~\cite{KSV} for $\mathcal{F} =
\{C_4,C_6,\dots,C_{2\ell}\}$ when $\ell \in \{2,3,5\}$, and again the extremal graphs are shown in~\cite{KSV} to be bipartite graphs: the bipartite incidence graphs of rank two
geometries called generalized polygons~\cite{Tits}. The case $\mathcal{F} = \{C_4\}$ with $k=5$ was settled earlier by Erd\H{o}s and Simonovits~\cite{ESi1}.

\bigskip

In the present paper, we give a general approach which proves the conjecture for all families of bipartite graphs which are known to be {\em smooth}
(in a sense which is defined below) using a sparse version of
Szemer\'{e}di's Regularity Lemma. It is convenient for the next definition to denote by $\zz{m}{n}{\mathcal{F}}$ the maximum number of edges in an $m$ by $n$ bipartite
$\mathcal{F}$-free graph. In what follows, all asymptotic notation is with respect to $n$. Let $\alpha,\beta$ be reals with $ 1 \leq \beta <\alpha<2$.

\begin{definition}
A family $\mathcal{F}$ of bipartite graphs is $(\alpha,\beta)$-smooth if for
some $\rho \ge 0$ and every $m \leq n$,
\[ \zz{m}{n}{\mathcal{F}} = \rho mn^{\alpha - 1} + O(n^{\beta}).\]
\end{definition}

It is a consequence of the K\"{o}v\'ari-S\'{o}s-Tur\'{a}n Theorem~\cite{KSTbound} that $\alpha < 2$ for every family $\mathcal{F}$ of bipartite graphs.
We say $\mathcal{F}$ is {\em smooth} if $\mathcal{F}$ is $(\alpha,\beta)$-smooth for
some $\alpha>\beta$, and $(\alpha,\beta)$-smooth with {\em relative density} $\rho$
when we wish to refer to the limit density.
If $\mathcal{F}$ is smooth and consists of a single graph $F$, then we use the terminology {\em  $F$ is smooth}.
Conjecture \ref{firstconj} perhaps suggests that every finite family of graphs
is smooth --  smoothness should be seen as a bipartite
version of Conjecture \ref{firstconj} together with some control on the error
term implicit in the limit. In this paper, we study smooth families of graphs
and use smoothness as a broad framework under which Conjecture \ref{ESC} might
be proved. Single bipartite graphs which are known to be smooth include
$K_{2,t}$ for $t \geq 2$ and $K_{3,3}$, however, even the cycle $C_6$ of length
six is not known to be smooth.

\subsection{Main Result}

A family $\mathcal{G}$ of graphs is {\em near-bipartite} if
every $n$-vertex graph $G \in \mathcal{G}$ has a bipartite subgraph $H$ with $e(G) \sim e(H)$ as $n \rightarrow \infty$.
To prove Conjecture \ref{ESC} for a family $\mathcal{F}$, it is enough to show every extremal $\mathcal{F} \cup \{C_k\}$-free
family of graphs is near-bipartite for some odd $k$, for then $\ex{n}{\mathcal{F} \cup \{C_k\}} \sim \z{n}{\mathcal{F}}$.
Our main theorem proves this strengthening of the Erd\H{o}s-Simonovits conjecture for smooth families of bipartite graphs:

\begin{thm}\label{thm:ExtCycle}
Let $\mathcal{F}$ be an $(\alpha,\beta)$-smooth family with $2 > \alpha > \beta \geq
1$. There exists $k_0$ such that if $k\ge k_0 \in \mathbb N$ is odd, then
every extremal $\mathcal{F} \cup \{C_k\}$-free family of graphs is
near-bipartite.
\end{thm}

The theorem settles Conjecture~\ref{ESC} in a strong sense for smooth families,
namely that the omission of any single odd cycle is enough to force extremal
$\mathcal{F}$-free families of graphs to be nearly bipartite, whereas the condition of
Conjecture~\ref{ESC} is to avoid all short odd cycles. The closer the value of
$\beta$ to $\alpha$, however, the larger the lower bound on $k$, the length of
the shortest odd cycle we can find. The proof of Theorem \ref{thm:ExtCycle} is a
novel application of a recent sparse version of Szemer\'{e}di's Regularity Lemma
due to Scott~\cite{ScottSparse}. We note that (despite the use of regularity)
the value of $k_0$ we obtain is not large -- for all the smooth families we
know, it is less than $20$. We also ask whether the extremal $n$-vertex $\mathcal{F} \cup \{C_k\}$-free graphs in
Theorem \ref{thm:ExtCycle} are exactly bipartite when $n$ is large enough.
So far this is not known, even for $\mathcal{F} = \{K_{2,t}\}$ with $t > 2$, which is $(3/2,4/3)$-smooth.
Here we give some evidence that extremal $\mathcal{F} \cup \{C_k\}$-free graphs are indeed bipartite, by
proving the following Andr\'{a}sfai-Erd\H{o}s-S\'{o}s type theorem, which verifies this
provided the minimum degree of an extremal graph is not too small.

\begin{thm}\label{thm:AES}
Let $\mathcal{F}$ be an $(\alpha,\beta)$-smooth family with $2 > \alpha > \beta \geq 1$.
Then there exists $k_0$ such that for all odd $k\ge k_0$ and $n$ sufficiently large the following is true.
If $G$ is an $n$-vertex $\mathcal{F} \cup \{C_k\}$-free graph
with minimum degree at least $\rho\big(\tfrac{2n}{5}+o(n)\big)^{\alpha-1}$
then $G$ is bipartite.
\end{thm}

\subsection{Complete Bipartite Graphs}

Theorem \ref{thm:ExtCycle} reduces the Erd\H{o}s-Simonovits Conjecture to showing that
a family of bipartite graphs is smooth. We shall now see that certain complete bipartite graphs are smooth.
F\"{u}redi~\cite{F2} showed that
if $m \le n$, and $s,t \in \mathbb N$, then
\[ \zz{m}{n}{K_{s,t}} \leq (t-s+1)^{1/s}mn^{1-1/s}+sm+sn^{2-2/s}\,.\]
If $s = 2$ or $t = s = 3$, this upper bound is matched up to a difference of $O(n^{2-2/s})$
 by constructions of Brown~\cite{B} and
F\"{u}redi~\cite{F2}, so in these cases, $K_{s,t}$ is $(2 - 1/s,2-2/s)$-smooth. It follows from Theorem \ref{thm:ExtCycle} that the Erd\H{o}s-Simonovits conjecture
holds for the
complete bipartite graphs
$K_{2,t}$ and $K_{3,3}$, and we obtain the following result:

\begin{thm}\label{thm:kst}
Let $k \geq 5$ be an odd integer. If $s = 2$ and $t \geq 2$, or $s = t = 3$, then
\[ \ex{n}{\{K_{s,t},C_k\}} \sim \z{n}{K_{s,t}} \sim (t - s + 1)^{1/s}(n/2)^{2 - 1/s}\]
and any extremal $\mathcal{F}$-free $C_k$-free graph is near-bipartite.
\end{thm}

This theorem is a direct consequence of Theorem \ref{thm:ExtCycle} and the smoothness
of $K_{s,t}$ for $s = 2$ and $s = t = 3$, except that the assumption $k \geq 5$ is weaker than the assumption on $k$
in Theorem \ref{thm:ExtCycle}
. For $t > (s - 1)!$, the order of magnitude  of
$\ex{n}{K_{s,t}}$ is known to be $n^{2 - 1/s}$ according to a construction of
Alon, R\'{o}nyai and Szab\'o~\cite{ARS}, superseding an earlier construction of
Koll\'{a}r, R\'{o}nyai and Szab\'o~\cite{KRS}. However, the asymptotic behavior
of
$\zz{m}{n}{K_{s,t}}$ is not known for any $s \geq 3$ and $t
\geq 4$, and the order of magnitude of $\ex{n}{K_{s,t}}$ is not known in
general.
In the event that $K_{s,t}$ is $(2 - 1/s,\beta)$-smooth for any given $s \leq
t$ and $\beta<2-1/s$, Theorem \ref{thm:kst} would extend to prove the
Erd\H{o}s-Simonovits conjecture for $K_{s,t}$ with the condition $k\ge 5$.

\subsection{Triangles and an old conjecture of Erd\H{o}s}

A natural question is whether it is also possible
that Theorems \ref{thm:ExtCycle} and \ref{thm:kst} could hold for shorter odd cycles, in particular for triangles -- the case $k = 3$ in both theorems.
In the case $s=t=2$, it is an old conjecture of Erd\H{o}s~\cite{ErdosA,ErdosB} that any extremal $n$-vertex graph of girth five
has $\frac{1}{2\sqrt{2}}n^{3/2} + o(n^{3/2})$ edges -- in particular the bipartite incidence graph of
a projective plane with $n$ points has that many edges. However, we are able to show that for
all odd $t\ge 3$, the extremal function $\ex{n}{\{K_{2,t},C_3\}}$ is not asymptotic to $\z{n}{K_{2,t}}$.

\begin{thm}\label{thm:k23} For each $t\in\mathbb{N}$,
  \[\liminf_{n\rightarrow\infty}\frac{\ex{n}{\{C_3,K_{2,2t+1}\}}}{\z{n}{K_{2,2t+1}}}
  \ge\frac{t+1}{\sqrt{t(t+2)}}\,.\]
\end{thm}

In particular the ratio is $\frac{2}{\sqrt{3}} + o(1)$ for $K_{2,3}$. In light of this, we make the opposite conjecture to Erd\H{o}s:

\begin{conj}
\[ \liminf_{n \rightarrow \infty}
\frac{\ex{n}{\{C_3,C_4\}}}{\z{n}{C_4}} > 1\,.\]
\end{conj}

In~\cite{KSV}, it is shown that Conjecture~\ref{ESC} does not extend to $k = 3$ when
$\mathcal{F} = \{C_4,C_6\}$ and does not extend to $k = 3, 5$ when $\mathcal{F} = \{C_4,C_6,C_8,C_{10}\}$.
Determining the smallest value of $k = k(\mathcal{F})$ for which the conjecture could hold is an interesting open problem.

\section{Sparse regularity}

In this section we describe a version of Szemer\'edi's Regularity Lemma which
can be usefully applied to extremal $\mathcal{F}$-free graphs.

\subsection{Szemer\'{e}di's Regularity Lemma}

Given an $n$-vertex graph $G$ and disjoint vertex subsets $X$ and $Y$, the
\emph{density} of the pair $(X,Y)$ is $d(X,Y)=e(X,Y)/(|X||Y|)$. We say that
$(X,Y)$ is \emph{$\eps$-regular} if we have $|d(X',Y')-d(X,Y)|\le\eps$ for every
pair $X'\subset X$ and $Y'\subset Y$ with $|X'|\ge\eps |X|$ and $|Y'|\ge\eps
|Y|$. We say a partition $V(G)=V_0\sqcup\cdots\sqcup V_k$ is
\emph{$\eps$-regular} if $|V_0|\le\eps n$, $|V_1|=\cdots=|V_k|$, and all but at
most $\eps k^2$ of the pairs $(V_i,V_j)$ with $i,j\in[k]$ are $\eps$-regular.
The celebrated Szemer\'edi Regularity Lemma~\cite{SzReg} states that for every
$\eps>0$ there exists $K$ such that every graph $G$ possesses an $\eps$-regular
partition whose number of parts is between $\eps^{-1}$ and $K$.
A key definition in the proof is the \emph{energy function} of a partition $\mathcal{P}$ of $V(G)$, which is
\begin{equation}\label{eq:DefEnergy}
  \mathcal{E}(G,\mathcal{P})=\sum_{X,Y\in\mathcal{P}}\frac{|X||Y|}{n^2}d(X,Y)^2.
\end{equation}

The proof consists of showing that either $\mathcal{P}$ is an $\eps$-regular partition of $G$, or there is a refinement of $\mathcal{P}$ with at most $2^{|\mathcal{P}|}$ parts
whose energy is increased by at least $\eps^5$ over that of
$\mathcal{P}$. Since the energy of any partition is bounded by $1$, the result
follows. Unfortunately, Szemer\'edi's Regularity Lemma in the above framework is useless for our purposes:
$K_{s,t}$-free graphs have $o(n^2)$ edges, and hence the densities of all pairs
in the partition will be very close to zero.

\subsection{Versions of the regularity lemma for sparse graphs}

For sparse graphs, Kohayakawa~\cite{YoshiSparse} and independently R\"odl
(unpublished) defined the notion of relative density and
$(\eps,p)$-regular partitions where $p \in [0,1]$. Here, $G$ is a graph and $X$ and $Y$ are disjoint non-empty subsets of vertices of $G$.

\begin{definition}
Given $0<p\le 1$, the \emph{relative density} of the pair
$(X,Y)$ is $d_p(X,Y)=d(X,Y)/p$. The pair $(X,Y)$ is \emph{$(\eps,p)$-regular}
if $|d_p(X',Y')-d_p(X,Y)|\le\eps$ for every
pair $X'\subset X$ and $Y'\subset Y$ with $|X'|\ge\eps |X|$ and $|Y'|\ge \eps |Y|$. A partition $V(G)=V_0\sqcup\cdots\sqcup V_k$ is called
\emph{$(\eps,p)$-regular} if $|V_0|\le\eps n$, $|V_1|=\cdots=|V_k|$, and all but at
most $\eps k^2$ of the pairs $(V_i,V_j)$ with $i,j \in [k]$ are
$(\eps,p)$-regular.
\end{definition}

Kohayakawa~\cite{YoshiSparse} and R\"odl proved that for any
function $p=p(n)$, $\eps>0$ and $C$ there exist $K$ and $\mu>0$ such that if $G$
is an $n$-vertex graph for which the relative density $d_p(X,Y)$ is at most $C$
whenever $X$ and $Y$ are disjoint vertex sets of size at least $\mu n$, then $G$
has an $(\eps,p)$-regular partition whose number of parts is between $\eps^{-1}$ and  $K$.
For this proof, the energy function
\begin{equation}\label{eq:DefpEnergy}
  \mathcal{E}_p(G,\mathcal{P})=\sum_{X,Y\in\mathcal{P}}\frac{|X||Y|}{n^2}d_p(X,Y)^2
\end{equation}
is defined, and is observed to be bounded by $C$ whenever $\mathcal{P}$ contains
no small parts; the remainder of the proof is virtually identical to the original proof of Szemer\'edi.
Unfortunately, $K_{s,t}$-free graphs with $\Theta(n^{2-1/s})$ edges and
$p(n)=n^{-1/s}$ need not satisfy the condition given by Kohayakawa and R\"odl.
For complete bipartite graphs this is not a major issue: one can use instead
the following bound on the energy function of a partition, from which the
desired regularity lemma follows.

\begin{lem}\label{lem:EnergyBound}
  Let $p=n^{-1/s}$. Let $G$ be a $K_{s,t}$-free graph and $\mathcal{P}$ a
  partition of $V(G)$ with no parts of size smaller than $2sn^{1/s}$. Then we
  have $\mathcal{E}_p(G,\mathcal{P})\le 2^st+1$.
\end{lem}
\begin{proof}
  Let $G$ be an $n$-vertex $K_{s,t}$-free graph, and $\mathcal{P}$ a partition
  of $V(G)$ with no parts containing fewer than $2sn^{1/s}$ vertices. Given
 part $X\in\mathcal{P}$ we have \[\frac{t|X|^s}{s!}\ge
t\binom{|X|}{s}\ge\sum_{v\in G}
    \binom{d_X(v)}{s}\ge\sum_{Y\in\mathcal{P}}|Y|\binom{d(X,Y)|X|}{s}
    =\sum_{Y\in\mathcal{P}}|Y|\binom{d_p(X,Y)n^{-1/s}|X|}{s}\, ,\]
  using that $G$ is $K_{s,t}$-free in the second step and Jensen's inequality at the third step. Observe that since
  $|X|>2sn^{1/s}$, if $d_p(X,Y)\ge 1$ then $d_p(X,Y)n^{-1/s}|X|-s\ge
  d_p(X,Y)n^{-1/s}|X|/2$. Ignoring terms with $d_p(X,Y)<1$, we have
  \[\frac{t|X|^s}{s!}\ge \sum_{\substack{Y\in\mathcal{P}\\d_p(X,Y)\ge
  1}}
|Y|\binom{d_p(X,Y)n^{-1/s}|X|}{s}\ge\sum_{\substack{Y\in\mathcal{P}\\d_p(X,Y)\ge
  1}}\frac{|Y|}{n}\frac{d_p(X,Y)^s|X|^s}{2^ss!}\]
  and thus
  \[
  \sum_{\substack{Y\in\mathcal{P}\\d_p(X,Y)\ge 1}}\frac{|Y|}{n}d_p(X,Y)^s\le
  2^st\,.\]

  Finally, since $x^2\le x^s$ for all $x\ge 1$ and $x^2<1$ for all $x<1$, we
  have
   \begin{align*}
\mathcal{E}_p(G,\mathcal{P})&=\sum_{X,Y\in\mathcal{P}}\frac{|X||Y|}{n^2}d_p(X,
Y)^2
\le \sum_{X\in\mathcal{P}}\frac{|X|}{n}\Big(\sum_{\substack{Y\in\mathcal{P}\\d_p(X,
Y)\ge
1}}\frac{|Y|}{n}d_p(X,Y)^s+\sum_{\substack{Y\in\mathcal{P}\\d_p(X,Y)<1}}\frac{
|Y|}{n}\Big)\\
    &\le \sum_{X\in\mathcal{P}}\frac{|X|}{n}\big(2^st+1)=2^st+1
  \end{align*}
  as required.
\end{proof}

However, Scott~\cite{ScottSparse} has recently shown that (perhaps surprisingly)
we need assume no conditions at
all on $G$ to obtain a sparse-regular partition, and it is convenient to simply
use this result to prove our main theorems. To the best of our knowledge, this is the first application of Scott's regularity
lemma.

\begin{lem}\label{lem:SparseReg} {\bf (Sparse Regularity Lemma)}
  Given $\eps>0$ and  $C\ge 1$, there exists $T = T(\eps)$ such that if $G$ is
any graph
  with at most $Cpn^2$ edges, then $G$ has an $(\eps,p)$-regular partition whose
number of parts is  between $\eps^{-1}$ and $T$.
\end{lem}

Briefly, Scott's proof rests on his observation that the Kohayakawa-R\"odl proof
fails for general graphs because it is possible for
$\mathcal{E}_p(G,\mathcal{P})$ to be large due to a small number of dense pairs
in $\mathcal{P}$, and thus no good enough bound on $\mathcal{E}_p$ for that
proof to work can hold. Scott's solution is to define a new energy function, replacing $d_p(X,Y)^2$ with
$\min\big(d_p(X,Y)^2,4L(d_p(X,Y)-L)\big)$ for some large constant $L$. This has
the
effect of `discounting' very dense pairs, so that the new energy function can be
proven to be bounded (by $CL^2$): of course, it also means that refining these
very dense pairs in a non-$(\eps,p)$-regular partition no longer
contributes to increasing the energy function, but, crucially, there are few
such dense pairs and thus an energy increase of $\eps^5/2$ is still guaranteed.
To facilitate use of Lemma~\ref{lem:SparseReg}, we give one last definition.

\begin{definition}
Given $d\in[0,1]$, the \emph{$(\eps,d,p)$-cluster graph} associated to an
$(\eps,p)$-regular partition $V_0\sqcup\cdots\sqcup V_t$ of $G$ is the
graph with vertex set $\{V_1,\ldots,V_t\}$ and edges $V_iV_j$ whenever
$(V_i,V_j)$ is an $(\eps,p)$-regular pair with $d_p(V_i,V_j)\ge d$.
\end{definition}

We shall also require some basic useful facts about regular pairs. These are proved
directly from the definition of regularity and density, and are standard facts
in the usual dense version of the regularity lemma (see~\cite{KomSimSurvey}).

\begin{prop}\label{prop:basic}
Let $(A,B)$ and $(A,C)$ be $(\eps,p)$-regular pairs of relative density at least $d$. Then
\begin{center}
\begin{tabular}{lp{5in}}
1. & $(A,B \cup C)$ is $(\eps,p)$-regular of relative density at least $d$. \\
2. & If $\gamma > \eps$ and $X \subseteq B$ has size at least $\gamma|B|$,
then $(A,X)$ is $(\eps+\eps/\gamma,p)$-regular of relative density at least $d -
\varepsilon$.
\end{tabular}
\end{center}
\end{prop}

\section{Expansion and embedding}

In this section, we show how to use the presence of many triangles in an
$(\eps,d,p)$-cluster graph of an $\mathcal{F}$-free graph $G$ to find odd cycles
in $G$. In the traditional setting of the regularity lemma, namely for dense
graphs, extremal theorems in the cluster graph are employed to find large structures in the
 original graph. In particular, a single clique $K_r$ in the cluster graph is sufficient to provide
 a relatively large complete $r$-partite subgraph of the original graph, which gives a
 derivation of the Erd\H{o}s-Stone Theorem from Tur\'an's Theorem via the regularity lemma. However, in the case of sparse graphs, the appearance  of a particular graph $H$ in the $(\eps,d,p)$-cluster graph is insufficient to imply a copy of $H$
 in the original graph. For example, a construction of Alon~\cite{AlonPseudoBad} shows that we cannot hope
to find even one triangle in a triple of sparse-regular pairs by making use only
of the sparse-regularity. Similarly, our example of a dense $K_{2,3}$-free triangle-free graph in Theorem \ref{thm:k23} is actually a triple of sparse-regular pairs, so the cluster graph is actually a triangle, whereas the original graph has no triangles as well as no $K_{2,3}$. More generally (see Krivelevich and Sudakov~\cite{KS}), one can take any pseudorandom $C_k$-free graph with $k$ odd and replace each vertex $v$ with a set $S_v$ of $o(n)$ new vertices and each edge $uv$ with the complete bipartite graph consisting of all edges between $S_u$ and $S_v$.
This still gives a pseudorandom $C_k$-free graph, which in particular has a
complete cluster graph,
whilst the original graph has no $C_k$. The issue of what further conditions
one should place in order to avoid these `bad examples' is an important
question. Recently Conlon, Fox and Zhao~\cite{CFZ} made substantial progress
in a quite general setting, but our problem is not covered by their work.

Our strategy to find odd cycles is then the following. We will see that, under
the assumption of smoothness, $\mathcal{F}$-free graphs which are close to
extremal (in the sense of having the right order of magnitude number of edges)
have reasonably strong expansion properties. This will allow us to find a
vertex $v$ such that the set $N_\ell(v)$ of vertices at distance $\ell$ from $v$
has linear size, and in the event that the cluster graph contains many triangles
we can further ensure that this set has substantial intersection with both
sides of a dense regular pair. By regularity there is then an edge $e$ in the
set, and thus a cycle of length $2\ell+1$ in the graph (here we ignore the
possibility that the two paths from the ends of $e$ to $v$ might intersect,
but this is not difficult to resolve).

\subsection{Expansion and embedding from smoothness}

The smoothness of a family $\mathcal{F}$ of bipartite graphs with exponent
$\alpha$ implies expansion in any $\mathcal{F}$-free graph with large minimum
degree. Precisely, we have the following lemma.

\begin{lem}\label{lem:expand}
Let $\mathcal{F}$ be an $(\alpha,\beta)$-smooth family of bipartite graphs with $2 > \alpha > \beta \geq 1$. Then there exists $C > 1$ such that for any $\delta > 0$, $\gamma = \delta/2C$ and any $\mathcal{F}$-free bipartite graph
$G(U,V)$ on at most $n \ge (1/\gamma)^{1/(\alpha-\beta)}$ vertices, if every
vertex in $U$ has degree at least $\delta n^{\alpha-1}$, then
\begin{enumerate}
 \item\label{lem:expand:a} if $|U| < |V|$, then $|V| \geq
  \min\big(\gamma^{1/\beta}|U|^{1/\beta}n^{(\alpha-1)/\beta},\gamma^{1/(\alpha - 1)}n\big)$.
 \item\label{lem:expand:b} if $|U| \geq |V|$, then $|V| \geq
\max\big(\gamma n^{\alpha-1}|U|^{2 - \alpha},\gamma^{1/(\alpha - 1)}n\big) $.
\end{enumerate}
\end{lem}

\begin{proof}
If $|U| < |V|$, then since $\mathcal{F}$ is smooth, there exists $C$ (which we
may insist is larger than $1$) such that
\[ \delta n^{\alpha-1}|U| \leq e(U,V) \leq C(|U||V|^{\alpha - 1} +
|V|^{\beta})\leq 2C\max\big(|U||V|^{\alpha-1},|V|^\beta\big).\]
This immediately implies part~\ref{lem:expand:a}. If $|U| \geq |V|$, then
\begin{equation} \label{part2bound}
\delta n^{\alpha-1}|U| \leq e(U,V) \leq C(|V||U|^{\alpha - 1} +
|U|^{\beta}).
\end{equation}
Since $n^{\alpha - \beta} \geq 1/\gamma$ and $|U| \leq n$ and $\beta \geq 1$,
\[ n^{\alpha - 1} \geq \frac{1}{\gamma}n^{\beta - 1} \geq \frac{2C}{\delta}|U|^{\beta - 1}.\]
We conclude $2C|U|^{\beta} \leq \delta n^{\alpha - 1}|U|$. Therefore (\ref{part2bound}) gives:
\[ \delta n^{\alpha - 1}|U| \leq 2C|V||U|^{\alpha - 1}.\]
This immediately gives $|V| \geq \gamma n^{\alpha - 1}|U|^{2 - \alpha}$.
By substituting $|U| \geq |V|$, and recalling $1<\alpha < 2$, we also obtain
\[ \delta n^{\alpha - 1}|V|^{2 - \alpha} \leq 2C|V|\]
which shows $|V| \geq \gamma^{1/(\alpha - 1)}n$, as required in part 2 of the lemma.
\end{proof}

One might reasonably think that in order to embed odd cycles, it will be enough
to find one triangle $ABB'$ in
the cluster graph. One would select a typical vertex $v$ in $A$, which has at
least $dp|B|/2$ neighbors in $B$, and similarly in $B'$: then
Lemma~\ref{lem:expand} (applied on $BB'$) guarantees a larger second
neighborhood of $v$ in
$B'$, and similarly in $B$: after $k$ steps we have subsets of $N_k(v)$ in
$B$ and $B'$ of linear size which we control by regularity. Unfortunately, the
expansion of $\mathcal{F}$-free
graphs is not quite strong enough for this argument to work: the minimum degree
depends upon the number of clusters, and is thus potentially much smaller than $\eps n^{\alpha-1}$. The result is that we can
only guarantee that $N_k(v)$ covers a fraction of a cluster which is $o(\eps)$, and $(\eps,p)$-regularity tells us
nothing about such small sets. Therefore we use an alternative approach, demanding not
one but a positive density of triangles through $A$ in the cluster graph. These triangles can be used
for an expansion argument in a large fraction of the graph, independent of the number of parts. We have then the following
embedding lemma.

\begin{lem}\label{lem:embed}
Let $\tau,d > 0$, $\alpha > \beta > 1$, and let $\mathcal{F}$ be an $(\alpha,\beta)$-smooth family of
bipartite graphs. Let $\ell_0 (\alpha,\beta)$ be
\[\ell_0 = \Big\lfloor\log_{\beta}\Bigl(\frac{(2\beta-\beta^2)(\alpha - 1)}{\alpha
-
\beta}\Bigr)\Big\rfloor\,~ \mbox{ for $\beta>1$ and}~~ \ell_0 = \big \lfloor1/(\alpha-1)\big\rfloor\,~ \mbox{for $\beta=1$}\,.\]
 Let $p = p(n) = n^{\alpha - 2}$ for $n \in \mathbb N$. Then there exist $\eps_0 =
\eps_0(\alpha,\tau,d)$ and $n_0 = n_0(\alpha,\beta,\tau,d)$ such that the
following holds. If
$\eps < \eps_0$, $n > n_0$, and $G$ is an $\mathcal{F}$-free graph with an
$(\eps,d,p)$-cluster graph $R$ with $t\ge\eps^{-1}$ clusters in which one
cluster $A$ is in at least $\tau t^2$ triangles, then $G$ contains $C_k$ for
each odd $k \in \mathbb
N$ satisfying
\[ 2\ell_0+5\le k\le \eps d n^{\alpha-1}/4t\,.\]
\end{lem}

\medskip

\begin{proof}
Our first task is to find a structure in $R$ suitable for our expansion method.
This will consist of clusters $A$, $B$, $B'$, sets of clusters $\mathcal{C}$
and $\mathcal{C}'$, and clusters $D_1,\ldots,D_{\ell_0}$ and
$D'_1,\ldots,D'_{\ell_0}$, which are adjacent as follows. The clusters
$D_1,\ldots,D_{\ell_0}$ form a path (in that order); $D_1$ is adjacent to $A$,
$D_{\ell_0}$ to all clusters $\mathcal{C}$, and all clusters $\mathcal{C}$ to
$B$. We require similar adjacencies involving
$B',\mathcal{C}',D'_1,\ldots,D'_{\ell_0}$. Finally $B$ and $B'$ are adjacent.
Our embedding will proceed as follows. We will choose a typical vertex $v$ in
$A$, which then has a substantial first neighborhood in both $D_1$ and $D'_1$.
We apply Lemma~\ref{lem:expand} between $D_i$ and $D_{i+1}$ for each $i$ to find
eventually that the $\ell_0$th neighborhood of $v$ in $D_{\ell_0}$ has size
very close to linear: in particular, large enough that at the next application
of Lemma~\ref{lem:expand} part 1 to bound the size of the $(\ell_0+1)$st
neighborhood of $v$ in $\mathcal{C}$ we will find that the second term in the
minimum is smaller. It is important that this neighborhood's size should not
depend upon $t$, and therefore at this step the minimum degree input to
Lemma~\ref{lem:expand} must be independent of $t$: to achieve this we will need
that $\mathcal{C}$ contains a large fraction (independent of $t$) of the entire cluster graph. In a final application of Lemma~\ref{lem:expand} we will find a $(\ell_0+2)$nd neighborhood of $v$ in $B$, and this will cover a fraction of $B$ independent
of $t$, and thus much bigger than $\eps$. By the same procedure we find a large
$(\ell_0+2)$nd neighborhood of $v$ in $B'$: by $(\eps,p)$-regularity of $BB'$
the bipartite graph between the two $(\ell_0+2)$nd neighborhoods contains many
edges, and in particular paths of all (not too long) lengths. By construction a
path of length $2i-1$ vertices in this graph extends to a cycle (through $v$) on
$2i+2\ell_0+3$ vertices, as required. This approach is illustrated in the figure below.

\begin{figure}
\psfrag{A}{$A$}
\psfrag{B}{$B$}
\psfrag{B'}{$B'$}
\psfrag{C}{$\mathcal{C}$}
\psfrag{C'}{$\mathcal{C}'$}
\psfrag{D1}{$D_1$}
\psfrag{Dl}{$D_{\ell_0}$}
\psfrag{D1'}{$D'_1$}
\psfrag{Dl'}{$D'_{\ell_0}$}
\psfrag{v}{$v$}
\psfrag{N1}{$N_1$}
\psfrag{N1'}{$N'_1$}
\psfrag{Nl}{$N_{\ell_0}$}
\psfrag{Nl'}{$N'_{\ell_0}$}
\psfrag{Nl1}{$N_{\ell_0+1}$}
\psfrag{Nl1'}{$N'_{\ell_0+1}$}
\psfrag{Nl2}{$N_{\ell_0+2}$}
\psfrag{Nl2'}{$N'_{\ell_0+2}$}
\centerline{\includegraphics[width=5in]{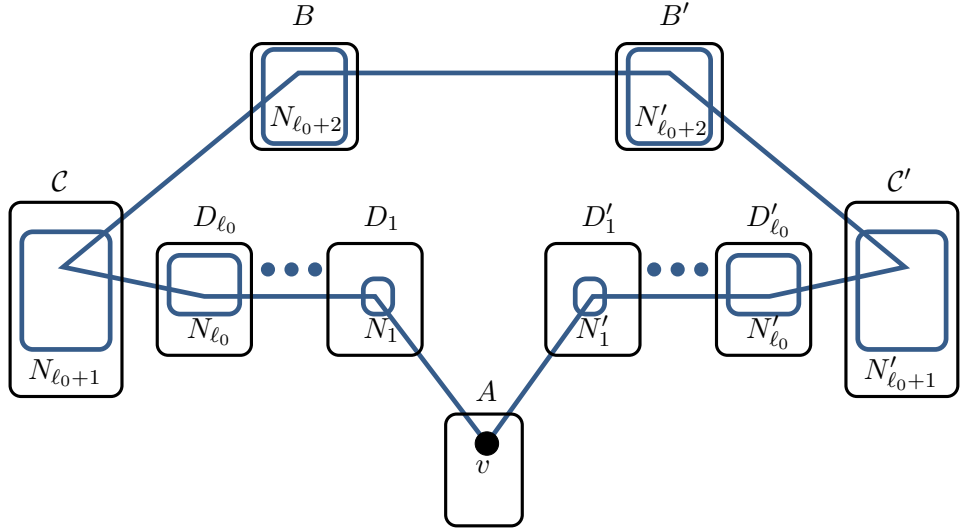}}
\caption{Constructing an odd cycle from the cluster graph}
\end{figure}

For the proof, we choose constants as follows. Let $C$ be the constant returned by
Lemma~\ref{lem:expand}. We will choose $\eps_0$ to be a small enough positive function
of $C,d,\ell_0,\tau$ and $\alpha$ in all the calculations that follow. Then we choose
$n_0$ to be large enough relative to these constants, $\epsilon$ and their reciprocals.

Let $H'$ be the subgraph of $R$
consisting of edges of $R$ which form triangles together with $A$: by definition
$H'$ has at least $\tau t^2$ edges, and so average degree at least $2\tau t$. It
follows that $H'$ has a subgraph $H$ whose minimum degree is at least $\tau t$,
which we fix.

We will now choose sequentially the clusters required for the structure described above,
under the assumption that every cluster has degree at least $\tau t/2$ in $H'$ to unused clusters.
Let $B$ and $B'$ be any two adjacent clusters in $H$. Let $\mathcal{C}_1$
consist of $\tau t/4$ neighbors of $B$ in $H$. Each of these clusters has at
least $\tau t/4$ neighbors not equal to $B$ or $B'$ or to clusters in $\mathcal{C}_1$, and
consequently there is a cluster $D_{\ell_0}$ adjacent to a set of $\tau^2 t/16$
clusters $\mathcal{C}$ of $\mathcal{C}_1$. We now construct the path
$D_{\ell_0},\ldots,D_1$ greedily choosing any so far unused neighbor of each
$D_i$ in turn. We now repeat the same process (avoiding previously used
clusters) to construct $\mathcal{C}'$ and $D'_{\ell_0},\ldots,D'_1$.
Note that at the end we have used only $2(\ell_0+1)+\tau^2 t/8<\tau t/2$
clusters (where the inequality follows from $t\ge\eps^{-1}$ and choice of
$\eps_0$), so the assumption of degree at least $\tau t/2$ to unused clusters is valid.

We now delete vertices from these clusters as follows (in a slight abuse of
notation, we will continue to use the same letters for the clusters after
deletion). Let
\[ \delta = \tau^2 d/64 \quad \text{ and } \quad \tilde{\delta} = d/4t.\]

 From each cluster in $\mathcal{C}$ we remove all vertices whose
degree into $B$ is smaller than $\tilde{\delta} n^{\alpha-1}$. Then
from $D_{\ell_0}$ we remove all vertices whose degree into $\bigcup\mathcal{C}$
is smaller than $\delta n^{\alpha-1}$, and then for
each $\ell_0-1\ge i\ge 1$ in succession we remove from $D_i$ all
vertices whose degree into $D_{i+1}$ is smaller than $\tilde{\delta}n^{\alpha-1}$.
Since we have $(d-\eps)>d/2$, by $(\eps,p)$-regularity of
the various pairs (including, by Proposition~\ref{prop:basic},
the pair $\big(D_{\ell_0},\bigcup\mathcal{C}\big)$ ) we do not delete more than
$\eps n/t$ vertices (and in
particular, since $\eps<1/2$, not more than half of the vertices) from any
cluster. We perform the same process on the clusters
$\mathcal{C}',D'_{\ell_0},\ldots,D'_1$.

We choose a vertex $v\in A$ whose neighborhood in both $D_1$ and $D'_1$ has
size at least $(d-\eps)n^{\alpha-1}/t> \tilde{\delta} n^{\alpha-1}$: again this is
possible by $(\eps,p)$-regularity of $AD_1$ and $AD'_1$, and since $\eps<1/2$.
Let $N_1$ be the neighbors of $v$ in $D_1$, and for each $2\le i\le \ell_0$
let $N_i$ be the neighbors of $N_{i-1}$ in $D_i$. Let $N_{\ell_0+1}$ be the
set of neighbors of $N_{\ell_0}$ in $\bigcup\mathcal{C}$, and finally $N_{\ell_0+2}$
be the set of neighbors of $N_{\ell_0+1}$ in $B$.

For $i \geq 0$ and $\beta \geq 1$, define $f(i,\beta) = \frac{1}{\beta - 1} + \frac{\beta^2 - 2\beta}{(\beta - 1)\beta^i}$.
Note that $f(i,\beta)$ is well defined since $\beta^i+\beta^2-2\beta$ is divisible by $\beta-1$ as a polynomial in $\beta$. Also, by definition,
$f(1,\beta)=1$ and $f(i,1) = i$ for all $i \geq 0$.
Let $\tilde{\gamma} = \tilde{\delta}/2C$, where $C$ is the constant in Lemma~\ref{lem:expand}. We claim that for
each $1\le i\le \ell_0$ we have
\[ |N_i| \geq \min\bigl\{\bigl(\tilde{\gamma}n^{\alpha - 1}\bigr)^{f(i,\beta)},\tilde{\gamma}^{1/(\alpha - 1)} n\bigr\}.\]
Proceed by induction on $i$. The statement is true for $i=1$ by choice of $v$. If the claim holds for some $1\le
i\le\ell_0-1$, then we apply Lemma~\ref{lem:expand} to the bipartite graph
spanned by edges between $U = N_i$ and $V = N_{i + 1}$. By construction, every vertex of $U$ has
degree at least $\tilde{\delta} n^{\alpha - 1}$. By Lemma~\ref{lem:expand} part 1,
if $|N_i| < |N_{i+1}|$ then
\[ |N_{i+1}| \geq \min\{\tilde{\gamma}^{1/\beta}|N_i|^{1/\beta}n^{(\alpha-1)/\beta},\tilde{\gamma}^{1/(\alpha - 1)}n\}\]
provided $n \geq n_0 \geq (2C/\tilde{\delta})^{1/(\alpha - \beta)}$. By the above lower bound on $|N_i|$,
the first expression in the minimum is at least
\[\tilde{\gamma}^{\tfrac{1}{\beta-1}+\tfrac{\beta^2-2\beta}{
(\beta-1)\beta^{i+1}}}n^{(\alpha-1)
\big(\tfrac{1}{\beta-1}+\tfrac{\beta^2-2\beta}{(\beta-1)\beta^{i+1}}\big)}\,,\]
as required. By Lemma~\ref{lem:expand} part 2, if $|N_i| \geq |N_{i+1}|$ then
$|N_{i+1}| \geq \tilde{\gamma}^{1/(\alpha - 1)}n$, as required. This verifies the claim.

We now apply Lemma~\ref{lem:expand} to the bipartite graph spanned by edges between
$U = N_{\ell_0}$ and $V = N_{\ell_0 + 1}$ to find bounds on the size of
$N_{\ell_0+1}$. Note that the degree of all vertices of $N_{\ell_0}$ into
$\bigcup \mathcal{C}$ is at least $\delta n^{\alpha - 1}$, by construction.
We have that for $\beta = 1$, $f(\ell_0 + 1,1) = \ell_0 + 1 > 1/(\alpha - 1)$.
Also for $\beta > 1$, the definition of $\ell_0$
ensures
$$f(\ell_0 + 1,\beta) = \frac{1}{\beta - 1}\left(1 - \frac{2\beta-\beta^2}{\beta^{\ell_0 + 1}}\right)
>\frac{1}{\beta - 1} \left(1 - \frac{2\beta-\beta^2}{\frac{(2\beta-\beta^2)(\alpha-1)}{\alpha-\beta}}\right)
=\frac{1}{\alpha - 1}.$$
Let $\gamma = \delta/2C$. Then Lemma~\ref{lem:expand} shows
$|N_{\ell_0 +1}| \geq \gamma^{1/(\alpha - 1)}n$. Indeed, we have that
the minimum in  part~\ref{lem:expand:a} of this lemma takes the
value $\gamma^{1/(\alpha - 1)}n$, since $f(\ell_0 + 1,\beta)>1/(\alpha - 1)$ and $n_0$ is large enough.
By choosing $\eps_0$ small enough and since $t\ge 1/\eps$, we have that $\gamma^{1/(\alpha - 1)}n >n/t$.
Thus we are guaranteed to
have $|N_{\ell_0+1}| > |B| \ge |N_{\ell_0+2}|$. Since every vertex of $N_{\ell_0 + 1}$ has
at least $\tilde{\delta}n^{\alpha - 1}$ neighbors in $B$, a final application of
Lemma~\ref{lem:expand} part~\ref{lem:expand:b} to the bipartite graph spanned by
edges between $N_{\ell_0 + 1}$ and $N_{\ell_0 + 2}$ with $\tilde{\gamma} = \tilde{\delta}/2C$
yields
\[|N_{\ell_0+2}| \ge \tilde{\gamma} n^{\alpha-1} \cdot \gamma^{\frac{2-\alpha}{\alpha -
1}}n^{2-\alpha}=\frac{d}{8Ct}n^{\alpha-1} \cdot \left(\frac{\tau^2d}{128C}\right)^{\frac{2-\alpha}{\alpha -
1}}n^{2-\alpha}=\frac{d}{8C} \left(\frac{\tau^2d}{128C}\right)^{\frac{2-\alpha}{\alpha -
1}}n/t
> \eps n/t\]
where the last inequality follows by choosing $\eps_0$ small enough.

We define similarly sets $N'_1,\ldots,N'_{\ell_0+2}$ and obtain that
$N'_{\ell_0+2}$ contains at least $\eps n/t$ vertices of $B'$. By
$(\eps,p)$-regularity the bipartite graph $G[N_{\ell_0+2},N'_{\ell_0+2}]$
has average degree at least $\eps (d-\eps)n^{\alpha-1}/t$ and so has a subgraph
with minimum degree at least $\eps d n^{\alpha-1}/4t$. This subgraph contains
paths of every odd length up to $\eps d n^{\alpha-1}/4t$. Since by our construction a path of length $2i-1$
between $B$ and $B'$ extends to an odd cycle in $G$ through $v$ of length $2i+2\ell_0+3$, this completes the proof.
\end{proof}

\vspace{0.1cm}
\noindent
{\bf Remark.} Note that, by definition, we have that $\ell_0 \geq 1$ for all values of $1 \leq \beta < \alpha < 2$.
Therefore this embedding lemma cannot be used to find odd cycles of length shorter than seven.

\subsection{Expansion and embedding: complete bipartite graphs}

The purpose of this section is to show that when $\mathcal{F}=\{K_{s,t}\}$ for
some $2\le s\le t$, if $K_{s,t}$ is $(2-1/s,\beta)$-smooth for some
$\beta<2-1/s$, then we can improve on the expansion and embedding of the
previous subsection, and embed cycles of length at least five. First we prove the
necessary expansion.

\begin{lem}\label{lem:kstexpand}
Let $G$ be an $n$-vertex $K_{s,t}$-free graph. Let $v$ be a vertex of $G$,
  $X\subset N(v)$, and $Y$ a subset of $V(G)$ disjoint from
  $X\cup\{v\}$. Suppose that $|X|\ge \rho_Xn^{1-1/s}$ and $|N(x)\cap Y|\ge
  \rho_Y|Y|n^{-1/s}$ for each $x\in X$. Then
  \[|N(X)\cap Y|\ge
  \Big(\frac{\rho_X}{t-1}\Big)^{\tfrac{1}{s-1}}\Big(\rho_Y|Y| - sn^{1/s}\Big).\]
\end{lem}
\begin{proof}
  Each vertex in $X$ contains at least $\binom{\rho_Y|Y|n^{-1/s}}{s-1}$ subsets
  of $N(X)\cap Y$ of size $s-1$ in its neighborhood. However, if there were
  $t$ vertices of $X$ all of whose neighborhoods contained one particular
  $(s-1)$-set in $N(X)\cap Y$, then together with $v$ we would find
  $K_{s,t}$ in $G$, which is forbidden. It follows that we have
  \[|X|\binom{\rho_Y|Y|n^{-\tfrac{1}{s}}}{s-1}\le (t-1)\binom{|N(X)\cap
  Y|}{s-1}\] and hence
  \begin{align*}
    |X|\big(\rho_Y|Y|n^{-\tfrac{1}{s}}-s+2\big)^{s-1}& \le
    (t-1)\big(|N(X)\cap Y|\big)^{s-1}
    \end{align*}
    which implies
  \begin{align*}
    |N(X)\cap Y| & \ge
(t-1)^{-\tfrac{1}{s-1}}\rho_X^{\tfrac{1}{s-1}}n^{\tfrac{1}{s}}\big(\rho_Y|Y|n^{-\tfrac{1}{s}}-s+2\big)\\
     & \ge\Big(\frac{\rho_X}{t-1}\Big)^{\tfrac{1}{s-1}}(\rho_Y|Y| - sn^{1/s}).
  \end{align*}
  This proves Lemma \ref{lem:kstexpand}.
\end{proof}

Our embedding statement is then the following.

\begin{lem}\label{lem:kstembed}
Let $\tau,d > 0$, and suppose that for some $2\le s\le t$ the graph $K_{s,t}$ is
$(2-1/s,\beta)$-smooth for some $\beta<2-1/s$.
Let $p
= p(n) = n^{-1/s}$ for $n \in \mathbb N$. Then there exist $\eps_0 =
\eps_0(\tau,d)$ and $n_0 = n_0(\tau,d)$ such that the following
holds. If
$\eps < \eps_0$, $n > n_0$, and $G$ is a $K_{s,t}$-free graph with an
$(\eps,d,p)$-cluster graph $R$ with $q\ge\eps^{-1}$ clusters in which one
cluster $A$ is in at least $\tau q^2$ triangles, then $G$ contains $C_k$ for
each odd $k \in \mathbb
N$ satisfying
\[ 5\le k\le \eps d n^{1-1/s}/4q\,.\]
\end{lem}
\begin{proof}[Sketch of proof]
We follow the notation and proof of Lemma~\ref{lem:embed}, to produce two adjacent clusters $B, B'$ and two families of
sets $\mathcal{C}$ and $\mathcal{C}'$ disjoint from them and from each other such that the following holds.
Every vertex of every set in $\mathcal{C}$ ($\mathcal{C}'$) has at least $dn^{1 - 1/s}/4q$
neighbors in $B$ ($B'$ respectively), each family
$\mathcal{C}$ and $\mathcal{C}'$ contains $\tau q/4$ sets, and every set in $\mathcal{C} \cup \mathcal{C}'$ was obtained from a partition cluster, which is  adjacent to $A$ in the
cluster graph, by removing fewer than $\eps n/q$ vertices. In particular, every set in $\mathcal{C} \cup \mathcal{C}'$ forms together with $A$ a $(3\eps,p)$-regular pair of relative density at least $d-\eps$.

Now we slightly modify the proof. Choose a vertex $v \in A$
which has at least $\tfrac{\tau^2d}{16} n^{1-1/s}$ neighbors in each of
$\bigcup\mathcal{C}$ and $\bigcup\mathcal{C}'$. This vertex exists by
$(3\eps,p)$-regularity of the pairs $\big(A,\bigcup\mathcal{C}\big)$,
$\big(A,\bigcup\mathcal{C}'\big)$, and since $\mathcal{C}, \mathcal{C}'$ each contain $\tau q/4$ sets and $\eps = o(d)$. We let $N_1$ be the
neighborhood of $v$ in $\bigcup\mathcal{C}$, and $N_2$ the neighborhood of
$N_1$ in $B$. By Lemma~\ref{lem:kstexpand} we have
\[|N_2|\ge
\Big(\frac{\tau^2d}{16(t-1)}\Big)^{\tfrac{1}{s-1}}(\tfrac{d}{4}|B|-sn^{1/s})
>\eps|B|\]
where the last inequality comes from choosing $\eps_0$ small enough and $n_0$ large enough. Similarly
we define $N'_1$ and $N'_2$ and obtain
that $|N'_2|>\eps|B'|$. Now since $B$ and $B'$ form an edge of the cluster graph,
the graph $G[N_2,N'_2]$ has average degree at least $\eps(d - \eps)n^{1 - 1/s}/q$.
Then it has a subgraph with minimum degree at least $\eps d n^{1 - 1/s}/4q$ and therefore contains
paths of every odd length up to $\eps d n^{\alpha-1}/4q$.
Since by our construction a path of length $2i-1$
between $N_2$ and $N_2'$ extends to an odd cycle in $G$ through $v$ of length $2i+3$, this completes the proof.
\end{proof}

\section{Transference of density}

Suppose that $G$ is an $\mathcal{F}$-free graph, and $R$ an
$(\eps,d,p)$-cluster graph for $G$. The aim of this section is to show that
under the assumption of smoothness, the density of $G$ (relative to the extremal
$\mathcal{F}$-free graphs) transfers to an (absolute) density of $R$: this
allows us to use classical extremal results to prove our main theorems.

\begin{lem}\label{lem:degtransfer}
  Let $\mathcal{F}$ be an $(\alpha,\beta)$-smooth family with relative density $\rho$ for some
  $2>\alpha>\beta\ge 1$ and $\rho>0$. For each $\gamma>0$ there exist $\eps_0,d_0>0$ such that
  for each $0<\eps\le\eps_0$, $0<d\le d_0$ and $T$ there is $n_0$ with the
  following property. Suppose $n\ge n_0$, $G$ is an $\mathcal{F}$-free
  graph on $n$ vertices, $p=n^{\alpha-2}$ and $R$ is an $(\eps,d,p)$-cluster graph with
  $\eps^{-1}\le t\le T$ vertices as obtained from Lemma~\ref{lem:SparseReg}.
  \begin{enumerate}
    \item\label{degtrans1} If $|E(G)|\ge (\mu^{\alpha-1}+\gamma) \rho
    p\frac{n^2}{2}$, then $|E(R)|\ge(\mu-\gamma)\frac{t^2}{2}$.
    \item\label{degtrans2} If a cluster $V_i\in V(R)$
  meets at least $(\mu^{\alpha-1}+\gamma)\rho n^{\alpha-1}|V_i|$ edges of $G$,
  then either $V_i$ is in at least $\gamma t$ irregular pairs, or
  $d_R(V_i)\ge(\mu-\gamma) t$.
  \end{enumerate}
\end{lem}
\begin{proof}
  We choose $d_0=\gamma\rho/2$,
  $\eps_0$ to be sufficiently small function of $\gamma, \rho, \alpha$ and
  we will require $n_0$ to be sufficiently large for all the calculations that follow.

  Let $n\ge n_0$, $0<d\le
  d_0$ and $0<\eps\le\eps_0$, and let $G$ be an $n$-vertex $\mathcal{F}$-free graph with
$(\mu^{\alpha-1}+\gamma) \rho p\frac{n^2}{2}$ edges. Let $R$ be an $(\eps,d,p)$-cluster
  graph for $G$ with $p = n^{\alpha - 2}$, as obtained from a partition
  $\mathcal{P}$ given by Lemma \ref{lem:SparseReg}, with clusters
  $V_1,V_2,\dots,V_t$ of size $m$, exceptional cluster $V_0$, and $1/\eps \leq t
  \leq T(\eps)$. For $i\in[t]$, let $I_i$ and $D_i$ respectively denote the set
  of irregular pairs including $V_i$ and the set of regular pairs of density at
  least $d$ including $V_i$. Thus $|D_i|$ is the degree of $V_i$ in $R$. Let
  $n_i = m|D_i| + m|I_i|$. Applying the smoothness of $\mathcal{F}$ to the
  bipartite subgraph $H$ of $G$ consisting of edges in pairs $(V_i,V_j) \in I_i
  \cup D_i$, we obtain \[ |E(H)| \leq \zz{m}{n_i}{\mathcal{F}} \leq \rho
  mn_i^{\alpha - 1} + O(n_i^{\beta}).\]
Since $\beta<\alpha$, provided $n$ is large enough, this is at
  most $(\rho + \eps)mn_i^{\alpha - 1}$. Therefore the total number of edges of
  $G$ meeting $V_i$ is at most
  \begin{equation}\label{eq:degtrans} e(V_i,V_0)+ e(V_i) + t \cdot dpm^2 + (\rho +
  \eps) m n_i^{\alpha - 1} \le \gamma \rho mn^{\alpha-1}+\rho m
  n_i^{\alpha-1}.\end{equation}
Here we used that $e(V_i) \leq 2\rho m^{\alpha}\le 2\rho m
  (\eps n)^{\alpha-1}$,  $e(V_i,V_0) \leq \rho m(\eps n)^{\alpha-1}+O((\eps
  n)^\beta)\le 2\rho m(\eps n)^{\alpha-1}$, $n$ is large enough, $\eps_0$ is sufficiently small,
$p=n^{\alpha-2}$ and $d \leq d_0=\gamma\rho/2$.

  Summing over all clusters, we obtain
  \[ 2|E(G)| \le \gamma \rho mt n^{\alpha-1} + \rho m \sum_{i = 1}^t n_i^{\alpha - 1}
  \le \gamma \rho n^\alpha +\rho mt\Big(\sum_{i = 1}^t
    \frac{n_i}{t}\Big)^{\alpha - 1}
  \] where the final part of the inequality is Jensen's inequality, using
  $\alpha-1<1$.
    On the other hand, $|E(G)| \geq (\mu^{\alpha-1}+\gamma)\rho
    p\frac{n^2}{2}=(\mu^{\alpha-1}+\gamma)\rho \frac{n^\alpha}{2}$. Putting these together and simplifying we have $ \mu n
    \leq \sum_{i = 1}^t \frac{n_i}{t}$,
    and thus
    \[\sum_{i=1}^t(|D_i|+|I_i|)\ge \mu tn/m\ge \mu t^2\,.\]
  Since $\mathcal{P}$ is a regular partition and $\eps = o(\gamma)$, the $|I_i|$ sum to at most $2\eps
  t^2<\gamma t^2$, and since the $|D_i|$ sum to $2|E(R)|$,
  part~\ref{degtrans1} follows.

  For part~\ref{degtrans2}, we have that the number of edges meeting $V_i$
  is at least $(\mu^{\alpha-1}+\gamma)\rho n^{\alpha-1}m$. This, together
  with~\eqref{eq:degtrans}, implies that $n_i\ge\mu n\ge\mu m t$. Thus
  $|D_i|+|I_i|\ge\mu t$, and so either $|I_i|\ge\sqrt{\eps}t$ or (since $\eps = o(\gamma)$)
  $|D_i|\ge(\mu-\gamma)t$, as desired.
\end{proof}

\section{Proof of Theorems~\ref{thm:ExtCycle} and~\ref{thm:kst}}

The extremal result which we transfer to the $\mathcal{F}$-free setting is
essentially the Stability Theorem of Simonovits~\cite{SimStab} for triangles.
However we require a slight strengthening: we forbid the existence of a vertex contained
in many triangles, rather than any triangle at all.

\begin{lem}\label{lem:TriStab}
  Given $0<\gamma<1/8$, let $G$ be an $n$-vertex graph with
  $e(G)\ge\big(\tfrac{1}{4}-\gamma\big)n^2$. Then either there is a vertex of
  $G$ contained in at least $\gamma n^2$ triangles, or we can delete at most
  $9\gamma^{1/4}n^2$ edges of $G$ to make it bipartite.
\end{lem}
\begin{proof}
Let $u$ be a vertex of maximum degree in $G$, $X$ be its neighborhood and let $Y=V(G)\setminus X$.
By definition of $u$, every vertex in $Y$ is incident with at most $|X|$ edges. Therefore the number of edges in $G$ satisfies
\[\big(\tfrac{1}{4}-\gamma\big)n^2\le e(G)\le
e(X)+|Y||X|=e(X)+(n-|X|)|X|.\]
Thus either $e(X)\ge \gamma n^2$, in which case we are done, or we have
$\big(\tfrac{1}{2}-\sqrt{2\gamma}\big)n\le |X|\le \big(\tfrac{1}{2}+\sqrt{2\gamma}\big)n$.

  If $e(Y)\le (9\gamma^{1/4}-\gamma)n^2$, then $G[X,Y]$ is the desired large
  bipartite subgraph of $G$. So suppose this is not the case. Let $v$ be a
  vertex of $X$ with maximum degree in $Y$, and $Z$ its neighborhood in $Y$. Either $e(Z)\ge \gamma n^2$, in which case we are done,
  or we have
  \[(9\gamma^{1/4}-\gamma)n^2\le e(Y)\le \binom{|Y|}{2}-\binom{|Z|}{2}+\gamma
  n^2.\]
From this, using that $|Y|+|Z| \leq 2|Y|\leq (1+2\sqrt{2\gamma})n$, we obtain
  \[(18\gamma^{1/4}-4\gamma)n^2\le |Y|(|Y|-1)-|Z|(|Z|-1)\le |Y|^2-|Z|^2\le
  \big(1+2\sqrt{2\gamma}\big)n(|Y|-|Z|).\]
Therefore $|Y|-|Z|\ge (18\gamma^{1/4}-4\gamma)n/2\ge 7\gamma^{1/4}n$. Since
we
  have $|Y|\le \big(\tfrac{1}{2}+\sqrt{2\gamma}\big)n$, we have
  \begin{equation}\label{eq:findZ}
    |Z|\le\big(\tfrac{1}{2}-5\gamma^{1/4}\big)n\,.
  \end{equation}
On the other hand, since
  no vertex of $X$ has more than $|Z|$ neighbors in $Y$, we have \[\big(\tfrac{1}{4}-\gamma\big)n^2\le e(G)\le e(X)+|X||Z|+e(Y)\le 2\gamma
  n^2+|X||Z|+\binom{|Y|}{2}-\binom{|Z|}{2}.\]
Together with  $|X|, |Y| \leq \big(\tfrac{1}{2}+\sqrt{2\gamma}\big)n$ and $|Z| \leq n$, this implies that
  \[\big(\tfrac{1}{4}-3\gamma\big)n^2\le
  \big(\tfrac{1}{2}+\sqrt{2\gamma}\big)n|Z|+\frac{\big(\tfrac{1}{2}+\sqrt{2\gamma}\big)^2n^2}{2}-\frac{|Z|^2}{2}
\leq \sqrt{2\gamma}n^2+\frac{\big(\tfrac{1}{2}+\sqrt{2\gamma}\big)^2n^2}{2}+\frac{(n-|Z|)|Z|}{2}
.\]
Simplifying, we get $\big(\tfrac{1}{4}-3\sqrt{2\gamma}-8\gamma\big)n^2\le
  \big(n-|Z|\big)|Z|$. Finally, this implies
  \[|Z|\ge
  \big(\tfrac{1}{2}-\sqrt{3\sqrt{2\gamma}+8\gamma}\big)n\ge\big(\tfrac{1}{2}-4\gamma^{1/4}\big)n\,,\]
  which contradicts~\eqref{eq:findZ} and completes the proof.
\end{proof}

We will in fact prove slightly more than the statement of
Theorem~\ref{thm:ExtCycle}: we will show that any family of
$\mathcal{F}\cup\{C_k\}$-free graphs in which the number of edges is
asymptotically extremal, is near-bipartite (rather than just proving it for the
extremal graphs themselves).

\begin{proof}[Proof of Theorem~\ref{thm:ExtCycle}]
Let $\mathcal{F}$ be an $(\alpha,\beta)$-smooth family and let $p=n^{\alpha-2}$.
By definition, we have that the maximum possible number of edges in a $\mathcal{F}$-free bipartite graph on $n$ vertices
is at most $(1+o(1))\rho (n/2)^\alpha=\big((1/2)^{\alpha-1}+o(1)\big)\rho p \frac{n^2}{2}$. Let
$\gamma>0$ be an arbitrarily small but fixed constant and let $\eps_0 \leq \gamma$ and $d\le d_0=\gamma\rho/2$ be sufficiently small for
  Lemma~\ref{lem:degtransfer}. Let $\eps\le\eps_0$ be sufficiently small for Lemma~\ref{lem:embed} with
  \[\tau:=\tfrac{1}{2}-\tfrac{1}{2}\big(1-2^\alpha
  \gamma\big)^{1/(\alpha-1)}+\gamma\,.\]
  Note that as $\gamma>0$ tends to zero, by definition $\tau>0$ also tends to
  zero. Let $T$ be the constant returned by Lemma~\ref{lem:SparseReg},
  and $n_0$ be large enough for Lemmas~\ref{lem:embed}
  and~\ref{lem:degtransfer}.

  Let $n\ge n_0$, and $G$ be an $n$-vertex $\mathcal{F}$-free graph with
  \[e(G)\ge
  \Big(\big(\tfrac{1}{2}\big)^{\alpha-1}-\gamma\Big)\rho
  p\frac{n^2}{2} = \Big(\big(\tfrac{1}{2}-\tau+\gamma\big)^{ \alpha-1 }
  +\gamma\Big)\rho  p\frac{n^2}{2}\,,\]
  where the equality comes from the definition of $\tau$.

  By Lemma~\ref{lem:SparseReg} there is a sparse-regular partition of $G$ with a $t$-vertex $(\eps,d,p)$-cluster graph $R$,
where $\eps^{-1} \le t \le T$.
 By Lemma~\ref{lem:degtransfer} (with $\mu=1/2-\tau+\gamma$) we have $e(R)\ge \big(\tfrac{1}{2}-\tau\big)\frac{t^2}{2}\geq \big(\tfrac{1}{4}-\tau\big)t^2$.
  If $R$ has a vertex in $\tau t^2$ triangles, then by Lemma~\ref{lem:embed} $G$ contains $C_k$ for each $k\geq k_0=2\ell_0+5$ with
$\ell_0$  specified in the statement of Lemma~\ref{lem:embed}, and we are done. If
  not, then by Lemma~\ref{lem:TriStab} $R$ has a subgraph $R'$ with at most
  $9\tau^{1/4}t^2$ edges such that $E(R)\setminus E(R')$ is bipartite.

Let $G'$ be the spanning subgraph of $G$ consisting of edges belonging to pairs in $E(R')$, pairs with density less than $d$, irregular pairs and edges inside partition classes.
Note that the $(\eps,p)$-regular partition of $G$ which gives the $(\eps,d,p)$-cluster graph $R$ is also
an $(\eps,p)$-regular partition of $G'$, and it gives the graph $R'$ as an $(\eps,d,p)$-cluster graph for $G'$.
Since $|E(R')| \leq 9\tau^{1/4}t^2$, it follows by Lemma~\ref{lem:degtransfer}, with $\mu=9\tau^{1/4}+\gamma$, that we
have $e(G') \leq \big((9\tau^{1/4}+\gamma)^{\alpha-1}+\gamma\big)\rho p \binom{n}{2}$. Since $\tau\rightarrow 0$ as
$\gamma\rightarrow 0$, this shows that $e(G')=o\big(e(G)\big)$ and since the graph with edges $E(G)\setminus E(G')$ is bipartite,
this proves that $G$ is near-bipartite.
\end{proof}

The proof of Theorem~\ref{thm:kst} is identical except that
Lemma~\ref{lem:kstembed} replaces Lemma~\ref{lem:embed}.

\section{Proof of Theorem \ref{thm:AES}}

We have shown that the extremal graphs for forbidding
$\mathcal{F}$ and $C_k$, for any smooth $\mathcal{F}$, are near-bipartite.
In this section we prove Theorem \ref{thm:AES}, which says indeed that extremal graphs are actually bipartite provided they have large enough minimum degree and sufficiently many vertices.

\begin{proof}[Sketch of proof]
Write $a(n) \ll b(n)$ for $a(n) = o(b(n))$ as $n \rightarrow \infty$. 
 Fix constants $0< \eps \ll \eps' \ll d \ll \tau\ll \gamma \ll 1$ so that all inequalities
required below are satisfied (for the purposes of this sketch we will not specify the constants).
 Let $\ell_0=\ell_0(\alpha, \beta)$ be the constant defined in Lemma \ref{lem:embed}.
 Let $k_0 = 2\ell_0+3$ and suppose $k \ge k_0$ is odd.
 We assume $n$ to be sufficiently large given the preceding constants, and set $p=n^{\alpha-2}$.
 Let $G$ be an $n$-vertex $\mathcal{F}$-free graph with minimum degree at least
 $\rho\big(\tfrac{2n}{5}+\gamma n\big)^{\alpha-1}$. Let $R$ be the associated
 $t$-vertex $(\eps,d,p)$-cluster graph generated by Lemma~\ref{lem:SparseReg}.
 By Lemma~\ref{lem:degtransfer}, all but at most $\eps' t$ clusters of $R$ have degree at least $(2/5+\gamma/2)t$.
Also, one can show that there are at most $\eps' t$ clusters that are `bad', in that they contain more than $\eps' n/t$ vertices that are incident to more than $2dpn$ edges in pairs that are irregular or have relative density less than $d$.
 Thus we may restrict to a subgraph $R'$ of $R$ with at least $(1-2\eps')t$ clusters, all of which are good,
 such that the minimum degree in $R'$ is at least $(2/5+\gamma/4)t$. Now Lemma 9 in~\cite{AESgen}, which is a
`supersaturated version' of the Andr\'asfai-Erd\H{o}s-S\'os theorem, guarantees
that either there is a vertex of $R'$ contained in $\tau t^2$ triangles,
or there is a partition $(\mathcal{X},\mathcal{Y})$ of $V(R')$ such that for each vertex of $R'$,
all but $\gamma t/16$ of its incident edges cross $(\mathcal{X},\mathcal{Y})$.
In the former case the proof is similar to that of Theorem \ref{thm:ExtCycle},
so we may suppose we are in the latter case.
Also note that $2t/5 \leq |\mathcal{X}|, |\mathcal{Y}| \leq 3t/5$.

Let $(X,Y)$ be a maximal cut of $G$ obtained from the partition
$(\bigcup\mathcal{X},V(G)\setminus\bigcup\mathcal{X})$ by successively moving
vertices with a majority of neighbors in their own partition class to the
other partition class. Using regularity and the fact that all clusters in $R'$ are good,
one can check that at most a $2\eps'$-fraction of any cluster in $V(R')$ will be moved in this
process. By maximality, each vertex of $G$ is incident to at least
$\frac{1}{2}\rho\big(\tfrac{2n}{5}\big)^{\alpha-1}$ edges crossing $(X,Y)$.

Since $G$ is not bipartite, suppose without loss of generality that $uv$ is an
edge in $X$. Let $N_1$ and $N'_1$ be disjoint, equal-sized sets of neighbors of
respectively $u$ and $v$ in $Y$. Let $N_2$ and $N'_2$ be disjoint, equal-sized
sets of neighbors of respectively $N_1$ and $N'_1$ in $X\setminus{u,v}$, and so
on. Using Lemma~\ref{lem:expand} and similar computation as in the proof of Lemma \ref{lem:embed},
it is not hard to check that it is possible to
choose these sets in such a way that after $\ell=\ell_0+1$ steps we have sets
$N_\ell$ and $N'_\ell$ which are both of size $\tau n$ and are disjoint.
Let $\mathcal{C}$, $\mathcal{C}'$ respectively be the set of clusters $C$
that contain at least $\tau |C|/2$ vertices of $N_\ell$, $N'_\ell$.
Note that $|\mathcal{C}| \ge (|N_\ell|-\tau n/2)/(n/t)= \tau t/2$,
and similarly $|\mathcal{C}'| \ge \tau t/2$.
Fix any $C \in \mathcal{C}$.
Since each vertex of $R'$ has at most  $\gamma t/16$ neighbors in its own part,
for any $C' \in \mathcal{C}$ there are at least
$(2/5+3\gamma/16)t + (2/5+3\gamma/16)t - 3t/5 \ge t/5$
clusters that are common neighbors of $C$ and $C'$.
Since $|\mathcal{C}'| \ge \tau t/2$, a simple double counting
argument shows that in particular there is a cluster $D$ of $R'$ which is
adjacent both to $C$ and to a set $\mc{D}$ of at least $\tau t/10$ clusters of $\mathcal{C}'$.
Let $Z$ be obtained from $N'_\ell \cap \cup_{C' \in \mc{D}} C'$
by deleting all vertices with degree less than $dn/4t$ in $D$.
Then $|Z| \ge \tau^3 n$.
Applying Lemma~\ref{lem:expand} part 2 between $Z$ and $D$, we reach a set $N'_{\ell+1}$ in
$D$ disjoint from all previous $N_i$ and $N'_i$, and of size exceeding $\eps |D|$.
By $(\eps,p)$-regularity of the pair $(C,D)$, there is a path between
$N_\ell \cap C$ and $N'_{\ell+1}$ of length $k-2\ell-2$.
By construction, this can be extended to a $k$-cycle in $G$.
\end{proof}

The value of $k_0$ in this theorem is similar in size to that in
Theorem~\ref{thm:ExtCycle}. In the case of the complete bipartite graphs
$K_{2,t}$ and $K_{3,3}$ (with
more care in the argument) we
can have $k_0=5$. Finally, if we impose the further
condition $e(G)\sim \rho(n/2)^\alpha$, then we can reduce the minimum degree to
$o(\rho p n)$ (in this case Theorem~\ref{thm:ExtCycle} shows that the cluster
graph of $G$ is almost bipartite). This reduces the problem of showing that
extremal
$n$-vertex $\mathcal{F}\cup C_k$-free graphs are bipartite to that of showing
that they do not have vertices of degree much smaller than the average.
Unfortunately, even for $\mathcal{F}=\{C_4\}$ this is not known for most values
of $n$.

\section{A new smooth family}

Let $B_t$ be the book with $t$ pages: that is, $t$ copies of $K_{2,2}$ sharing a
common edge (and no other vertices). In this section, we give an example
of an application of Theorem \ref{thm:ExtCycle} which determines the asymptotic
behavior of the Tur\'{a}n numbers $\ex{n}{\{K_{2,t},B_t,C_k\}}$ for $k \geq 9$.
For the proof, we require one result on counting $4$-cycles
in graphs, which will be used to show $\{K_{2,t},B_t\}$ is a smooth family.

\begin{lem}\label{lem:Supersat2}
Let $G$ be a bipartite graph with parts $U$ and $V$ of sizes $m$ and $n$ and
$e$ edges, and suppose that $e(e - n) \geq nm(m-1)/2$. Then the number of
$4$-cycles in $G$ is at least
\[ \frac{e^2(e-n)^2-e(e-n)nm(m-1)}{4n^2m(m-1)}\,.\]
\end{lem}
\begin{proof}
Let $N := \sum_{v \in V} {d(v) \choose 2}$.
Note that $N = \sum_{\{u,v\} \subset U} d(u,v)$,
where $d(u,v)$ is the number of common neighbors of $u$ and $v$.
The number of $4$-cycles in $G$ is exactly
$\sum_{\{u,v\} \subset U} {d(u,v) \choose 2}$.
By Jensen's inequality, this is at least
\[ \frac{m(m-1)}{2} {\frac{2}{m(m-1)} \sum_{\{u,v\} \subset U} d(u,v) \choose
2} = \frac{m(m-1)}{2} {\frac{2N}{m(m-1)} \choose 2} = \frac{N^2}{m(m-1)} -
\frac{N}{2}\,.\]
Applying Jensen's inequality a second time, we have
\[N\ge n\binom{e/n}{2}=\frac{e(e-n)}{2n}.\]
Since $e(e - n) \geq nm(m-1)/2$,
we deduce $N \geq m(m-1)/4$. For these values of $N$,
$N^2/m(m-1) - N/2$ is non-decreasing, and therefore the number of
$4$-cycles in $G$ is at least
\[ \frac{N^2}{m(m-1)} - \frac{N}{2} \geq
\frac{e^2(e-n)^2-e(e-n)nm(m-1)}{4n^2m(m-1)}\,.\]
\end{proof}

Now we show that $\{K_{2,t},B_t\}$ is smooth.

\begin{thm}\label{thm:booksmooth} For each $t\ge 2$, the family
$\{K_{2,t},B_t\}$ is $(\tfrac{3}{2},1)$-smooth. Consequently, for $k \geq 9$,
\[ \ex{n}{\{K_{2,t},B_t,C_k\}} \sim \z{n}{K_{2,t},B_t} \sim (n/2)^{3/2}.\]
\end{thm}
\begin{proof}
The asymptotic formula will follow from Theorem \ref{thm:ExtCycle} and
the known lower bound on the number of edges in a $K_{2,2}$-free bipartite graph, namely
\[ \z{n}{\{K_{2,t},B_t\}} \geq \z{n}{K_{2,2}} \geq \bigl(\tfrac{n}{2}\bigr)^{3/2}-O(n).\]
Next we require an upper bound on $\zz{m}{n}{\{K_{2,t},B_t\}}$. We shall show
\[ \zz{m}{n}{\{K_{2,t},B_t\}} < mn^{1/2} + 4t^2 n.\]
Let $G$ be a $\{K_{2,t},B_t\}$-free bipartite graph with parts $X$ and $Y$ of
sizes $m$ and $n$, where $m\le n$. Suppose for a contradiction, that $G$ has at least $mn^{1/2} + 4t^2 n$ edges.
We may delete edges of $G$ to obtain a subgraph $H$ of $G$ with exactly $e := mn^{1/2} + 4t^2 n$ edges. Clearly, $H$ is also $\{K_{2,t},B_t\}$-free.
Let $xy$ be any edge of $H$, and consider
the subgraph $H_{xy}$ of $H$ induced by
$N(x)\setminus\{y\},N(y)\setminus\{x\}$. If this subgraph has a
vertex of degree $t-1$, then $H$ contains $K_{2,t}$; if it has a matching of
$t$ edges, then $H$ contains $B_t$. It follows that since $H$ is
$\{K_{2,t},B_t\}$-free, then $H_{xy}$ has at most $2(t-1)(t-2)$ edges, as a
maximum matching has at most $2(t-1)$ vertices each incident to at most $t-2$
edges. Now the number of $4$-cycles in $H$ is precisely
$\tfrac{1}{4}\sum_{xy} e(H_{xy})$, which is thus at most $(t-1)(t-2)e(H)/2$.
Since
\[ e(e - n) = (mn^{1/2} + 4t^2n)(mn^{1/2} + (4t^2 - 1)n) \geq (mn^{1/2})^2 = nm^2\]
we may apply Lemma~\ref{lem:Supersat2} to conclude that the number of $4$-cycles in $H$ is
at least
\[\frac{(mn^{1/2}+3t^2n)^4-(mn^{1/2}+4t^2n)^2nm^2}{4n^2m^2}\ge
t^2mn^{1/2}+2t^4n > \frac{t^2}{2}(mn^{1/2} + 4t^2 n)> (t - 1)(t - 2)\frac{e(H)}{2}, \]
which is a contradiction. This shows that $\{K_{2,t},B_t\}$ is $(\tfrac{3}{2},1)$-smooth.
Hence we can use Theorem \ref{thm:ExtCycle} (in the form it was proved in Section 5) with $k_0=2\ell_0+5$ and $\ell_0=\lfloor1/(\alpha-1)\rfloor=2$
to complete the proof.
\end{proof}

\bigskip

We remark that one could use the methods of the above proof to give new asymptotics for Tur\'{a}n numbers of various other families of bipartite graphs plus an odd cycle.

\section{Proof of Theorem \ref{thm:k23}}

F\"{u}redi~\cite{F2} showed that
\[ \zz{n}{n}{K_{2,3}} \sim \sqrt{2}n^{3/2}.\]
In this section, we describe a construction of a tripartite graph $G_q$ with no triangle or $K_{2,3}$
having $n=3q^2$ vertices and
\[ e(G_q) = 3q^2(q-1) = \frac{1}{\sqrt{3}}n^{3/2}-n,\]
for any prime $q = 2$ mod $3$.
This construction is denser than any possible bipartite construction by a multiplicative factor of $2/\sqrt{3}$,
as for $n$ even, $\zz{n/2}{n/2}{K_{2,3}} \sim \frac{1}{2}n^{3/2}$, and
so Theorem \ref{thm:k23} is proved for $t = 1$.
Note that the exact value $\zz{n}{n}{K_{2,3}}$ is known only for finitely many
$n$.
We also remark that the purpose of choosing $q = 2$ mod $3$ is to make $-3$ a non-residue mod $q$.
This follows from the law of quadratic reciprocity, which states that if $p$ and $q$
are odd primes, and we define $p^*$ to equal $p$ if $p = 1$ mod $4$
or $-p$ if $p = 3$ mod $4$, then $q$ is a quadratic residue mod $p$
if and only if $p^*$ is a quadratic residue mod $q$. Applying this with $p=3$
we see that $-3$ is a quadratic residue mod $q$ if and only if $q$ is a quadratic
residue mod $3$, i.e. $q = 1$ mod $3$.

\medskip

{\bf Construction of $G_q$.} Let $G_q$ have parts $A_1$, $A_2$ and $A_3$ which are copies of $\mb{F}_q \times \mb{F}_q$.
Join $(x_1,x_2) \in A_i$ to $(y_1,y_2) \in A_{i+1}$ (where $A_4:=A_1$)
if $(x_1,x_2)-(y_1,y_2)=(a,a^2)$ for some $a \in \mb{F}_q^*$.

Clearly $G_q$ has $3q^2$ vertices and $3q^2(q-1)$ edges.

\medskip

{\bf Claim 1.} {\it $G_q$ is triangle-free.}

{\it Proof of Claim 1.} Suppose for a contradiction that we have a triangle in $G_q$, with vertices
$(x_1,x_2) \in A_1$, $(y_1,y_2) \in A_2$ and $(z_1,z_2) \in A_3$. Then
$(x_1,x_2)-(y_1,y_2)=(a,a^2)$, $(y_1,y_2)-(z_1,z_2)=(b,b^2)$, $(z_1,z_2)-(x_1,x_2)=(c,c^2)$
for some $a,b,c \in \mb{F}_q^*$.
It follows that $a+b+c=0$ and $a^2+b^2+c^2=0$, so
\[ 0 = 2(a^2+b^2+c^2)-(a+b+c)^2=(a-b)^2+c^2-2(a+b)c=(a-b)^2+3c^2.\]
This implies that $-3 = \left( \frac{a-b}{c} \right)^2$ is a quadratic residue.
But this contradicts our assumption that $q = 2$ mod $3$, so $G_q$ is triangle-free.

\medskip

{\bf Claim 2.} {\it $G_q$ is $K_{2,3}$-free.}

\medskip

{\it Proof of Claim 2.} First we claim that each bipartite graph $G_q(A_i,A_{i+1})$
is $C_4$-free. If not, then we have distinct vertices
\[ (x_1,x_2), (x'_1,x'_2) \in A_i \quad  (y_1,y_2), (y'_1,y'_2) \in A_{i+1}\]
and $a,b,c,d \in \mb{F}_q^*$ such that
\begin{eqnarray*}
(x_1,x_2)-(y_1,y_2)=(a,a^2) &\quad& (x'_1,x'_2)-(y'_1,y'_2)=(b,b^2) \\
(x_1,x_2)-(y'_1,y'_2)=(c,c^2) &\quad& (x'_1,x'_2)-(y_1,y_2)=(d,d^2).
\end{eqnarray*}
Then $a+b=c+d$ and $a^2+b^2=c^2+d^2$,
so
\[ (a-b)^2 = 2(a^2+b^2)-(a+b)^2 = 2(c^2+d^2)-(c+d)^2 = (c-d)^2.\]
Now we either have $a-b=c-d$, which gives $a=c$ and $b=d$,
or $a-b=d-c$, which gives $a=d$ and $b=c$. Either way we contradict the fact that we chose
distinct vertices above, so $G_q(A_i,A_{i+1})$ is a $C_4$-free graph.
It follows that any pair of vertices in a given $A_i$ has at most one common neighbor
in each $A_j$ with $j \ne i$. Thus in any potential $K_{2,3} \subset G_q$ the part of $K_{2,3}$ of size $2$ must
have vertices in different parts of $G_q$. Suppose, without loss of generality, that we have a vertex in $A_0$
and a vertex in $A_1$ that have $3$ common neighbors in $A_2$. Then we obtain
distinct $a,c,e \in \mb{F}_q^*$ and distinct $b,d,f \in \mb{F}_q^*$ such that
$a+b=c+d=e+f$ and $a^2+b^2=c^2+d^2=e^2+f^2$. As above, this implies
$(a-b)^2=(c-d)^2=(e-f)^2$. On taking square roots, each of the three terms in the
identity have two possible signs. We can assume without loss of generality that the
first two terms get the same sign, i.e. $a-b=c-d$. This gives $a=c$ and $b=d$,
which contradicts distinctness. Therefore $G_q$ is $K_{2,3}$-free.

\subsection{Construction of $K_{2,2t+1}$-free graphs}

Here we generalize the construction in the previous subsection. For any $t \ge 1$ and
certain primes $q$ we construct $(t+2)$-partite graphs $G_{q,t}$ with no triangle or $K_{2,2t+1}$
having $n=(t+2)q^2$ vertices and
\[ e(G_{q,t}) = \binom{t+2}{2}q^2(q-1) = \frac{t+1}{2\sqrt{t+2}} n^{3/2} - \frac{1}{2}(t+1)n.\]
Since F\"{u}redi~\cite{F2} showed that
\[ \zz{n/2}{n/2}{K_{2,2t+1}} \le \frac{\sqrt{t}}{2}n^{3/2}+n/4\]
this construction is denser than any bipartite construction by an asymptotic multiplicative factor of
$(t + 1)/\sqrt{t(t + 2)}$, as required for Theorem \ref{thm:k23}.

\medskip

{\bf Construction of $G_{q,t}$.}  Fix a prime $q$ and `multipliers' $m_{i,j} \in \mb{F}_q$
for $1 \le i<j \le t+2$ satisfying certain conditions to be described below.
Let $G_{q,t}$ have parts $A_1$, $\dots$, $A_{t+2}$ which are copies of $\mb{F}_q \times \mb{F}_q$.
For $i<j$, join $(x_1,x_2) \in A_i$ to $(y_1,y_2) \in A_j$
if $(x_1,x_2)-(y_1,y_2)=m_{i,j}(c,c^2)$ for some $c \in \mb{F}_q^*$.

Clearly $G_{q,t}$ has $(t+2)q^2$ vertices and $\binom{t+2}{2}q^2(q-1)$ edges.
We will derive conditions on the $m_{i,j}$ which ensure that $G_{q,t}$
has no triangle or $K_{2,2t+1}$, and then show that the conditions can be satisfied.
We adopt the convention that $m_{i,j}=-m_{j,i}$ when $1 \le j<i \le t+2$.

\medskip

{\bf Claim 1.} {\it $G_{q,t}$ is triangle-free if for all distinct $i,j,k$,
\[ m_{i,j,k}:=-m_{i,j}m_{j,k}m_{k,i}(m_{i,j}+m_{j,k}+m_{k,i})\]
is a quadratic non-residue in $\mathbb F_q$.
}

\medskip

{\it Proof of Claim 1.} Suppose $G_{q,t}$ has a triangle with vertices
$(x_1,x_2) \in A_i$, $(y_1,y_2) \in A_j$, $(z_1,z_2) \in A_k$, $a,b,c \in \mb{F}_q^*$
with $(x_1,x_2)-(y_1,y_2)=m_{i,j}(a,a^2)$, $(y_1,y_2)-(z_1,z_2)=m_{j,k}(b,b^2)$,
$(z_1,z_2)-(x_1,x_2)=m_{k,i}(c,c^2)$. Then $m_{i,j}a+m_{j,k}b+m_{k,i}c=0$
and $m_{i,j}a^2+m_{j,k}b^2+m_{k,i}c^2=0$, so
\begin{eqnarray*}
0 &=& (m_{i,j}+m_{j,k})(m_{i,j}a^2+m_{j,k}b^2+m_{k,i}c^2)-(m_{i,j}a+m_{j,k}b+m_{k,i}c)^2 \\
&=& m_{i,j}m_{j,k}(a-b)^2 + (m_{i,j}+m_{j,k}-m_{k,i})m_{k,i}c^2 - 2(m_{i,j}a+m_{j,k}b)m_{k,i}c \\
&=& m_{i,j}m_{j,k}(a-b)^2 + (m_{i,j}+m_{j,k}+m_{k,i})m_{k,i}c^2\,,
\end{eqnarray*}
and hence
\[-(m_{i,j}+m_{j,k}+m_{k,i})m_{i,j}m_{j,k}m_{k,i}=m_{i,j}^2m_{j,k}^2(a-b)^2/c^2
\,.\]
Thus we can make $G_{q,t}$ triangle-free if we arrange that
$m_{i,j,k}=-m_{i,j}m_{j,k}m_{k,i}(m_{i,j}+m_{j,k}+m_{k,i})$ is a non-residue
for all distinct $i,j,k$.
(Note that $m_{i,j,k}$ does not depend on the order of $i,j,k$.) This proves Claim 1.

\medskip

{\bf Claim 2.} {\it $G_{q,t}$ is $K_{2,2t + 1}$-free if $m_{j,k} \ne -m_{k,i}$ for all distinct $i,j,k$.}

\medskip

{\it Proof of Claim 2.} Each pair $(A_i,A_j)$ induces a $C_4$-free bipartite graph, so any pair of vertices
in a given $A_i$ has at most one common neighbor in each $A_j$ with $j \ne i$.
Thus in any potential $K_{2,2t+1}$ the part of size $2$ must use two different parts,
say we have $x \in A_i$ and $y \in A_j$ with $2t+1$ common neighbors.
There are $t$ parts besides $A_i$ and $A_j$, so one of them, say $A_k$,
contains at least $3$ common neighbors of $x$ and $y$. Then we obtain
distinct $a,c,e \in \mb{F}_q^*$ and distinct $b,d,f \in \mb{F}_q^*$ such that
\[ m_{j,k}a+m_{k,i}b=m_{j,k}c+m_{k,i}d=m_{j,k}e+m_{k,i}f\]
and
\[ m_{j,k}a^2+m_{k,i}b^2=m_{j,k}c^2+m_{k,i}d^2=m_{j,k}e^2+m_{k,i}f^2.\]
Since
\[ m_{j,k}m_{k,i}(a-b)^2 = (m_{j,k}+m_{k,i})(m_{j,k}a^2+m_{k,i}b^2)-(m_{j,k}a+m_{k,i}b)^2\]
this gives $(a-b)^2=(c-d)^2=(e-f)^2$. On taking square roots, we can assume
without loss of generality that the first two terms get the same sign, i.e. $a-b=c-d$,
so $a-c=b-d$. We also have $m_{j,k}(a-c)=-m_{k,i}(b-d)$, so if $m_{j,k} \ne -m_{k,i}$
we get $a=c$ and $b=d$, which contradicts distinctness.
Thus we can make $G_{q,t}$ be $K_{2,2t+1}$-free if we arrange that
$m_{j,k} \ne -m_{k,i}$ for all distinct $i,j,k$. This proves Claim 2.

\medskip

It remains to show that we can choose multipliers $m_{i,j}$ satisfying
the conditions in Claims 1 and 2. Intuitively, this should be easy
for fixed $t$ and large $q$ by choosing random multipliers: then we will be unlikely
to have $m_{j,k} = -m_{k,i}$, and given $m_{i,j}$, $m_{j,k}$, the choice of $m_{k,i}$
will be equally likely to make $m_{i,j,k}$ a residue or a non-residue. In fact,
it is even possible to choose multipliers greedily to satisfy the conditions,
which we shall do.
To simplify our task we choose $q > 2^{t^4}$,
arbitrarily relabel the multipliers as $m_1,m_2,\dots,m_T$,
where $T=\binom{t+2}{2}$, and choose them so that each $m_i$ is a residue,
we do not have $m_i = m_j$ or $m_i = -m_j$ with $i \ne j$, and
$-(m_i+m_j+m_k)$ is a non-residue for all distinct $i,j,k$: since the product
of a residue and a non-residue is a non-residue, this will suffice.

\medskip

We can easily choose residues $m_1$ and $m_2$ with $m_1 \ne m_2$ and $m_1 \ne -m_2$.
Suppose inductively that we have chosen $m_1,m_2,\dots,m_r$ satisfying these
conditions for some $2 \le r < T$. We use Weil's inequality in the
following form
(see, e.g., \cite[Chapter 5]{LN}):

\medskip

\begin{prop}\label{Weil}
Suppose $\chi$ is the quadratic character, i.e. $\chi(x)$ is $1$ if $x$ is a residue,
$-1$ if $x$ is a non-residue, or $0$ if $x=0$. Suppose also that $f \in \mb{F}_q[X]$
is a polynomial of degree $d \ge 1$ that is not the square of another polynomial. Then
$\left|\sum_{x \in \mb{F}_q}\chi(f(x))\right| \le (d-1)\sqrt{q}$.
\end{prop}

Let $I=\{-m_i-m_j: 1 \le i<j \le r\}$ and $g(x) = \prod_{m \in I}
(1-\chi(m-x^2))$.
Note that if $m-x^2$ is a non-residue for all $m \in I$ then $g(x)=2^{|I|}$.
On the other hand, if $m-x^2$ is a residue for any $m \in I$ then $g(x)=0$.
There are at most $2|I|$ values of $x$ such that $m-x^2$ is zero for some $m
\in I$
and each of these gives a value of $g(x)$ between $0$ and $2^{|I|-1}$.
Let $X$ be the set of $x$ such that $m-x^2$ is a non-residue for all $m \in
I$.
Then we have
\[ \Bigl|\sum_{x \in \mb{F}_q} g(x) - 2^{|I|}|X|\Bigr| \leq 2^{|I|}|I|.\]
We can also expand the product in the definition of $g$ to get
\[ \sum_{x \in \mb{F}_q} g(x) = q + \sum_{\emptyset \ne S \sub I} (-1)^{|S|}
\sum_{x \in \mb{F}_q} \prod_{m \in S} \chi(m-x^2).\]
Applying Weil's inequality to the functions
\[ f_S = \prod_{m \in S} \chi(m-x^2)\]
we have
\[ \Bigl|\sum_{x \in \mb{F}_q} g(x) - q\Bigr| \leq \sum_{\emptyset \ne S \sub I}
(2|S|-1)\sqrt{q}\le 2^{|I|}(2|I|-1)\sqrt{q}\,.\]
We deduce that $|2^{|I|}|X| - q| \leq 2^{|I|}|I|+2^{|I|}(2|I|-1)\sqrt{q}$.
Since $q>2^{t^4}$ we have $|X| > 4r$, so we can choose $m_{r+1}=x^2$ for some $x \in X$
such that $m_{r+1}$ is not equal to $m_i$ or $-m_i$ for any $1 \le i \le r$.
Therefore we can choose the multipliers $m_1,m_2,\dots,m_T$
satisfying the conditions in Claims 1 and 2.

\section{Concluding remarks}

\noindent $\bullet$ In this paper we give a different and new approach to the Erd\H{o}s-Simonovits conjecture,
employing Scott's sparse version of Szemer\'{e}di's Regularity Lemma. This approach
could potentially be used to give broader structural understanding of extremal
$\mathcal{F}$-free graphs when $\mathcal{F}$ consists of bipartite graphs
together with graphs of chromatic number $r$. In the present paper, we deal
with $r = 3$. More generally, let $\mathcal{C}_k^r$ denote the family of all
graphs of chromatic number $r$ with at most $k$ vertices, and $\mathcal{C}^r$
denote the family of all graphs of chromatic number $r$.
Generalizing the Erd\H{o}s-Simonovits conjecture, we pose the following problem:

\begin{prob}\label{prob:relturan} Let $\mathcal{F}$ be a family of bipartite graphs. Determine whether there exists an integer $k$ such that
\[\ex{n}{\mathcal{F}\cup\mathcal{C}_k^r} \sim \ex{n}{\mathcal{F}\cup\mathcal{C}^r}\,.\]
\end{prob}

This would suggest that the extremal density of an $\mathcal{F}$-free graph without any small
$r$-chromatic subgraphs is asymptotic to the extremal density of an $\mathcal{F}$-free
$(r-1)$-partite graph. A key obstacle is to find an appropriate embedding lemma
analogous to
Lemma \ref{lem:embed}.

\noindent $\bullet$ Exact results are rare for bipartite Tur\'an problems.
Even for $C_4$, while the exact result for $\zz{n}{n}{C_4}$ is relatively straightforward,
F\"uredi's exact result for $\ex{n}{C_4}$ is an intricate argument.
No exact result is known for $\zz{n}{n}{K_{2,3}}$ (except for finitely many $n$).
Any construction achieving the natural upper bound would provide a positive
answer to the famous biplane problem, which asks if there are infinitely many $n$
such that there is a family $A_1,\cdots,A_n$ of subsets of an $n$-element set
satisfying $|A_i \cap A_j|=2$ for $i \ne j$.
Suppose that there is a bipartite graph with parts $X$ and $Y$ of size $n$ such that
every pair of vertices in $X$ has exactly $2$ neighbors in $Y$.
Then the set of all neighborhoods of vertices in $X$ is a biplane on $Y$.

\noindent $\bullet$ Unsurprisingly, it is not in general easy to show that any
particular family $\mathcal{F}$ is smooth. Besides $K_{2,t}$ and $K_{3,3}$, we
show that the family $\{K_{2,t},B_t\}$ is smooth for all $t\ge
2$ (where $B_t$ is the book of $4$-cycles with $t$ pages). We cannot at present
prove, however, that $B_t$ itself is smooth. We intend to investigate smooth
families in future work.

\end{document}